\documentclass[12pt,final]{amsart}
\usepackage{amsmath,amscd}
\usepackage{amssymb}
\usepackage{amsthm}
\usepackage{comment}
\usepackage{mathtools}
\usepackage{graphicx, xcolor}
\usepackage{geometry}
\usepackage{mathrsfs}
\usepackage[ocgcolorlinks, linkcolor=blue]{hyperref}

\usepackage{bm}
\usepackage{bbm}
\usepackage{url}

\usepackage[utf8]{inputenc}
\usepackage{mathtools,amssymb}
\usepackage{esint}
\usepackage{tikz}
\usepackage{dsfont}
\usepackage{relsize}
\usepackage{url}
\usepackage{xcolor}
\usepackage{graphicx}
\usepackage{mathrsfs}
\usepackage[shortlabels]{enumitem}
\usepackage{lineno}
\usepackage{amsmath}
\usepackage{enumitem}
\usepackage{amsthm} 
\usepackage{verbatim}
\usepackage{dsfont}
\numberwithin{equation}{section}

\allowdisplaybreaks

\mathtoolsset{showonlyrefs}

\graphicspath{{images/}}

\newtheorem{theorem}{Theorem}[section]
\newtheorem{lemma}[theorem]{Lemma}
\newtheorem{definition}[theorem]{Definition}

\newtheorem{proposition}[theorem]{Proposition}
\newtheorem{question}{Question}
\newtheorem{remark}[theorem]{Remark}

\title[Entanglement principle for the fractional Laplacian]{Entanglement principle for the fractional Laplacian with applications to inverse problems}

\author[A. Feizmohammadi]{Ali Feizmohammadi}
\address{Department of Mathematics, University of Toronto, 3359 Mississauga Road, Deerfield Hall, 3015, Mississauga, ON, Canada L5L 1C6}
\curraddr{}
\email{ali.feizmohammadi@utoronto.ca}

\author[Y.-H. Lin]{Yi-Hsuan Lin}
\address{Department of Applied Mathematics, National Yang Ming Chiao Tung University, Hsinchu, Taiwan \& Fakult\"at f\"ur Mathematik, University of Duisburg-Essen, Essen, Germany}
\curraddr{}
\email{yihsuanlin3@gmail.com}


\newcommand{\C}{{\mathbb C}}
\newcommand{\R}{{\mathbb R}}
\newcommand{\Z}{{\mathbb Z}}
\newcommand{\N}{{\mathbb N}}

\newcommand{\eps}{\epsilon}

\newcommand {\p} {\partial}

\newcommand{\LC}{\left(}
\newcommand{\RC}{\right)}
\newcommand{\wt}{\widetilde}

\newcommand{\norm}[1]{\lVert #1 \rVert}
\newcommand{\abs}[1]{\left\lvert #1 \right\rvert}
\DeclareMathOperator{\supp}{supp} 

\begin{document}

	\maketitle
	
	\begin{abstract}
		We prove an entanglement principle for fractional Laplace operators on $\mathbb R^n$ for $n\geq 2$ as follows; if different fractional powers of the Laplace operator acting on several distinct functions on $\mathbb R^n$, which vanish on some nonempty open set $\mathcal O$, are known to be linearly dependent on $\mathcal O$, then all the functions must be globally zero. This remarkable principle was recently discovered to be true for smooth functions on compact Riemannian manifolds without boundary \cite{FKU24}. Our main result extends the principle to the noncompact Euclidean space stated for tempered distributions under suitable decay conditions at infinity. We also present applications of this principle to solve new inverse problems for recovering anisotropic principal terms as well as zeroth order coefficients in fractional polyharmonic equations. Our proof of the entanglement principle uses the heat semigroup formulation of fractional Laplacian to establish connections between the principle and the study of several topics including interpolation properties for holomorphic functions under certain growth conditions at infinity, meromorphic extensions of holomorphic functions from a subdomain, as well as support theorems for spherical mean transforms on $\mathbb R^n$ that are defined as averages of functions over spheres. 
	\end{abstract}

	\tableofcontents

	\section{Introduction}\label{sec: introduction}

	Fractional Laplace operators are a well-known example of nonlocal operators that satisfy a surprising \emph{unique continuation property} (UCP); if $u\in H^{r}(\R^n)$ for some $r\in \R$, and if $u$ and its fractional power of Laplacian of some order $s\in (0,1)$, namely $(-\Delta)^s u$, both vanish on some nonempty open set, then $u$ must vanish globally on $\mathbb R^n$, see e.g. \cite{GSU20}. We also refer the reader to \cite{Riesz} for a classical result with stronger assumptions on $u$; see also \cite{Fall01022014,ruland2015unique,Yu17} for related results. An analogous (UCP) as above has been derived in \cite{CMR20} for the higher-order fractional Laplacian $(-\Delta)^s$ with $s\in (-\frac{n}{2},\infty) \setminus \Z$. The above (UCP) with $s\in (0,1)$ was further extended in \cite{GLX} to the case of the fractional Laplace–Beltrami operators $(-\Delta_g)^s$ on $\R^n$ with a smooth Riemannian metric $g$. We also mention the recent work \cite{kenig2024fractional} that derives (UCP) results for certain classes of variable coefficient fractional dynamical Schr\"odinger equations.  A common technique in derivation of (UCP) results for fractional Laplace operators is the Caffarelli--Silvestre extension procedure \cite{Caffarelli08082007} together with Carleman estimates from \cite{ruland2015unique}, see also \cite{ghosh2021non} for an alternative proof using heat semigroups. The above-mentioned (UCP) has been a key tool in solving inverse problems for certain classes of nonlocal equations. We refer the reader to \cite{GSU20} for the first result in this direction which subsequently led to significant research on inverse problems for nonlocal equations. This will be further discussed in Section~\ref{sec_ip_applications}.

	\subsection{Entanglement principle for the fractional Laplace operator}
	
	In this paper, we are partly concerned with establishing (UCP) for \emph{fractional polyharmonic operators} on $\R^n$. Precisely, let $N\geq 2$ and let $\mathcal O\subset \R^n$ be a nonempty open set. Suppose that $u\in H^{r}(\R^n)$ for some $r\in \R$ and that there holds
	\begin{equation}\label{UCP_poly}
		u|_{\mathcal O}= \sum_{k=1}^N b_k ((-\Delta)^{s_k} u)|_{\mathcal O} =0,
	\end{equation}
	for some $\{b_k\}_{k=1}^N \subset \C\setminus \{0\}$ and some $\{s_k\}\subset (0,\infty)\setminus \N$. Does it follow that $u=0$ on $\R^n$? 
	
	Let us mention that such operators are physically motivated by some probabilistic models; see e.g. \cite[Appendix B]{DLV21}. To the best of our knowledge, no prior results address the above (UCP) formulated in this generality. The explicit Caffarelli-Silvestre extension procedure \cite{Caffarelli08082007} for representing fractional Laplace operators as Dirichlet-to-Neumann maps for degenerate elliptic equations has been a key tool in the study of (UCP) for single-term fractional Laplace operators (see e.g. \cite{ruland2015unique,GSU20}). Such explicit representations are not known for fractional polyharmonic operators. In addition, approaches based on heat semigroup representations of fractional Laplace operators face several technical difficulties, arising from the fact that multiple nonlocal terms contribute to the expression \eqref{UCP_poly} and isolating the terms is not feasible. In this paper, we establish (UCP) for \eqref{UCP_poly} as a particular case of a much broader principle that we refer to as the {\em entanglement principle} for fractional Laplace operators, stated as the following broad question.
	\begin{question}\label{question}
		Let $N\in \N$, let $\{s_k\}_{k=1}^N\subset (0,\infty)\setminus \N$ and let $\mathcal{O}\subset \R^n$ be a nonempty open set. Let $\{u_k\}_{k=1}^N$ be sufficiently fast decaying functions at infinity and assume that 
		\begin{equation}\label{ent_u_cond}
			u_1|_{\mathcal O}=\ldots=u_N|_{\mathcal O}=0 \quad \text{and} \quad  \sum_{k=1}^N  b_k((-\Delta)^{s_k}u_k)\big|_{\mathcal O}=0,
		\end{equation}
		for some $\{b_k\}_{k=1}^N\subset \C\setminus \{0\}$. Does it follow that $u_k\equiv 0$ in $\R^n$ for all $k=1,\ldots, N$?
	\end{question}
	When $N=1$, the above question has an affirmative answer, as it reduces to the well-known (UCP) for the fractional Laplace operator. However, for $N\geq 2$, this is a much stronger statement than (UCP), since it involves several distinct functions simultaneously in one equation. The nomenclature of the principle comes from \cite[Theorem 1.8]{FKU24} where, among other theorems proved in that paper, the authors discovered the entanglement principle for fractional Laplace-Beltrami operators on closed Riemannian manifolds, i.e. compact Riemannian manifolds without boundary. We thus aim to extend that principle to the case of Euclidean spaces. The main difference here lies in the noncompactness of the Euclidean space $\R^n$ which, as we will discuss later in Section~\ref{sec_outline_proof}, creates several important difficulties; see also \cite[Remark 1.9]{FKU24} on why compactness of the ambient manifold is an important feature there. We will affirmatively answer the above question under suitable decay rates for $\{u_k\}_{k=1}^{N}$ at infinity together with an additional assumption for the fractional exponents $\{s_k\}_{k=1}^N$. To state our result, we first need to define the notion of \emph{super-exponential decay at infinity} for a distribution on $\mathbb R^n$ as follows.
	
	\begin{definition}[Super-exponential decay at infinity]
		\label{def_exp}
		Let $u\in H^{-r}(\mathbb R^n)$ for some $r\in \R$. We say that $u$ has super-exponential decay at infinity if there exist constants $C,\rho>0$ and $\gamma>1$ such that given each $R>0$ there holds
		\begin{equation}\label{super-exponential decay weak}
		    |\langle u, \phi\rangle| \leq C e^{-\rho R^\gamma} \|\phi\|_{H^{r}(\mathbb R^n)}, \quad \text{for all } \phi \in C^{\infty}_0(\mathbb R^n\setminus B_R(0)).
		\end{equation}
		Here, $\langle \cdot,\cdot\rangle$ is the continuous extension of the Hermitian $L^2(\R^n)$-inner product as a sesquilinear form to $H^{-r}(\R^n)\times H^{r}(\R^n)$ and $B_R(0)$ is the closed ball of radius $R>0$ centered at the origin in $\R^n$.
	\end{definition}
    
	To answer Question \ref{question}, we need to impose the following additional assumption on $\{s_k\}_{k=1}^N$: 

   \begin{enumerate}[\textbf{(H)}]
       \item\label{exponent condition} 
       We assume $\{s_k\}_{k=1}^N \subset (0,\infty)\setminus \N$ with $s_1<s_2<\ldots <s_N$ and that 
	\begin{equation}
		\begin{cases}
			s_k-s_j \notin \Z \quad &\text{for all $j\neq k$,} \quad \quad\text{if the dimension $n$ is even}\\
			s_k -s_j\notin \frac{1}{2}\Z \quad &\text{for all $j\neq k$,} \quad \quad \text{if the dimension $n$ is odd}.
		\end{cases}
	\end{equation}
   \end{enumerate}

	Our main result may be stated as follows, which will be proved in Section~\ref{sec: entanglement}. 
    
	\begin{theorem}[Entanglement principle]\label{thm: ent}
		Let $\mathcal{O}\subset \R^n$, $n\geq 2$, be a nonempty bounded open set and let $\{s_k\}_{k=1}^N$ satisfy \ref{exponent condition}. Assume that $\{u_k\}_{k=1}^N\subset H^{-r}(\R^n)$ for some $r\in \R$ and that its elements exhibit super-exponential decay at infinity in the sense of Definition~\ref{def_exp}. If, 
		\begin{align}\label{condition_UCP_u}
			u_1|_{\mathcal O}=\ldots=u_N|_{\mathcal O}=0 \quad \text{and} \quad  \sum_{k=1}^N  (b_k(-\Delta)^{s_k}u_k)\big|_{\mathcal O}=0,
		\end{align}
		for some $\{b_k\}_{k=1}^N\subset \C\setminus \{0\}$, then $u_k\equiv 0$ in $\R^n$ for each $k=1,\ldots,N$.
	\end{theorem}
	
	\begin{remark}
    Let us make some remarks about the optimality of the assumptions in the above theorem:
		\begin{enumerate}[(i)]
		    \item 
		 The assumption \ref{exponent condition} that for all $k\neq j$, we have that $s_k-s_j \notin \Z$  is optimal in the sense that there are counterexamples in its absence, i.e., there would exist functions $\{u_k\}_{k=1}^N$ fulfilling \eqref{condition_UCP_u}, but with $u_k \not \equiv 0$, for some $k=1,\ldots, N$. We refer the reader to \cite[Remark 1.10]{FKU24} for the construction of such counterexamples. However, as can be seen in the statement of the theorem, when the dimension is odd, we impose an additional requirement that $s_k-s_j$ is not an odd multiple of $\frac{1}{2}$. This assumption may not necessarily be optimal and may be an artifact of our proof. In odd dimensions, when two exponents $s_k$ and $s_j$ differ by an odd multiple of $\frac{1}{2}$, the corresponding terms $(-\Delta)^{s_k} u_k$ and $(-\Delta)^{s_j} u_j$ appear to create a resonance effect in our analysis which does not allow us to disentangle them from each other. For the sake of clarity of our presentation, we impose this additional assumption in odd dimensions. We believe that this condition may be removable with some further analysis.
		
	\item	The super-exponential decay at infinity \eqref{super-exponential decay weak} imposed in the theorem may also not be optimal. In fact, for all but one step in our proof (see Proposition~\ref{prop_reduction from SM transform}), it suffices to assume that $\{u_k\}_{k=1}^{\infty}$ have Schwartz decay at infinity. It appears to us that some decay assumption on the distributions is unavoidable unless one were to use an entirely new methodology. Fortunately, when it comes to applications of our entanglement principle to inverse problems as well as Runge approximation properties of solutions to fractional polyharmonic equations, we always only need to work with functions that are compactly supported in $\R^n$ and as such the super-exponential decay stated in the theorem will be satisfied in its applications related to inverse problems. We will present the key ideas in the proof of Theorem~\ref{thm: ent} as well as a comparison with \cite[Theorem 1.8]{FKU24} in Section~\ref{sec_outline_proof}.
        \end{enumerate}
	\end{remark}

	\subsection{Applications to inverse problems}
	\label{sec_ip_applications}
	
	Let us first give a brief overview of the previous literature of inverse problems for nonlocal equations. In \cite{GSU20}, it was discovered that the (UCP) property for $(-\Delta)^s u$, $s\in (0,1)$, can be used as a key tool to solve certain inverse problems for nonlocal equations. There, the authors showed that it is possible to determine an unknown function $q\in L^{\infty}(\Omega)$ from the knowledge of the so-called exterior Dirichlet-to-Neumann mapping 
	$$ C^{\infty}_0(W_1)\ni f\mapsto \left. (-\Delta)^su \right|_{W_2} \in H^{-s}(W_2),$$
	where $u\in H^{s}(\R^n)$ is the unique solution to
	\begin{equation}
		\begin{cases}
			(-\Delta)^s u+q(x) u =0 &\text{ in }\Omega, \\
			u=f &\text{ in }\Omega_e:=\R^n\setminus \overline{\Omega}
		\end{cases}
	\end{equation}
 and $W_1, W_2\Subset \Omega_e$ are two nonempty open sets. This inverse problem may be viewed as a nonlocal version of the well-known isotropic Calder\'{o}n problem in electrical impedance tomography.  The connection between inverse problems for certain nonlocal equations and their strong (UCP) has since led to significant research in this direction. We mention several examples of related works. The works \cite{GRSU20,cekic2020calderon} investigate (UCP) under low regularity assumptions and recovery of lower order coefficients from finite numbers of exterior measurements. In \cite{GLX,CLL2017simultaneously,KLZ-2022,GU2021calder}, the authors investigate inverse problems for nonlocal variable coefficient operators. The works \cite{CMR20,CMRU20} study inverse problems for higher-order fractional operators, and \cite{KRZ-2023,lili24b,LL2022inverse,lin2020monotonicity} consider nonlocal equations in the presence of additional nonlinear lower order terms. In the articles \cite{LLR2019calder,KLW2021calder,ruland2018exponential,RS17}, the authors derived stability results for similar inverse problems as well as studying inverse problems for certain evolution-type nonlocal equations involving both space and time. Very recently, a different perspective has also been employed by using the Caffarelli-Silvestre type reduction formula to obtain some uniqueness results for inverse problems related to nonlocal PDEs by reducing them to inverse problems for local PDEs, see e.g. \cite{CGRU2023reduction,LLU2023calder,LZ2024approximation}.
	
	Let us mention also that a new direction of research was initiated in the recent works \cite{feizmohammadi2024fractional_closed,FGKU_2025fractional} for solving nonlocal versions of the well-known and widely open anisotropic Calder\'{o}n problem stated on closed Riemannian manifolds. These works use entirely different properties of fractional Laplace--Beltrami operators and do not rely on (UCP).  We refer the reader to the follow-up works \cite{Chien23,choulli2023fractional,FKU24,lili24a,lin2024fractional,LLU2022para,Quan24,ruland2023revisiting} in this direction. 
    
    Before characterizing our first result on inverse problems, let us briefly motivate it by recalling an equivalent global formulation of the anisotropic Calder\'{o}n problem in the setting of Euclidean space $\R^n$ using Cauchy datasets. The problem goes back to the pioneering paper of Calder\'{o}n in \cite{calderon}.  
    Let $\Omega\subset \R^n$ be a bounded Lipschitz domain with $n\geq 2$ and suppose that $A(x)=(A^{jk}(x))_{j,k=1}^n$ is a real-valued positive definite symmetric matrix on $\R^n$ that is equal to identity in the exterior of $\Omega$. Consider the Cauchy dataset 
    $$
    \mathscr S_A=\{ \left(u, \nabla\cdot(A\nabla u)\right)\big|_{\Omega_e}\,:\, u\in C^{\infty}(\R^n)\quad\text{and}\quad \nabla\cdot(A\nabla u)=0 \quad \text{in $\Omega$}\,\}.
    $$
    The anisotropic Calder\'{o}n problem is equivalent to the question of determining an a priori unknown matrix $A$ inside $\Omega$ from the knowledge of $\mathscr S_A$. It was noted by Tartar (account given in \cite{KV1984_Tartar}) that uniqueness is possible only modulo a gauge as follows. If $\Phi:\R^n\to \R^n$ is a diffeomorphism that fixes the exterior region $\Omega_e$, then there holds:
    $$  \mathscr S_{A} = \mathscr S_{B},\quad \text{where}\quad  B(x)= \frac{(D\Phi)A(D\Phi)^t}{|\det(D\Phi)|}\bigg|_{\Phi^{-1}(x)}\quad \forall\,x\in \R^n.$$
    Here, $D\Phi$ is the Jacobean matrix of $\Phi$. We mention that the anisotropic Calder\'{o}n problem has been solved in dimension two \cite{ASENS_2001_4_34_5_771_0,dd646efa-ce8c-38c2-804a-74a1c777c7c0} but remains widely open in higher dimensions outside the category of a real-analytic $A$ (which was solved in \cite{https://doi.org/10.1002/cpa.3160420804}). We refer the reader to \cite{SLSEDP_2012-2013____A13_0} for a survey of the anisotropic Calder\'{o}n problem.

	
	\subsection{Recovery of an anisotropic principal order term} For our first result, we solve a variant of the anisotropic Calder\'{o}n problem discussed above in the presence of lower order nonlocal terms in the equation. As we will see, in dimensions three and higher, we can determine the matrix $A$ uniquely without any gauge. To describe the result,  let us start with the setup. Let $\Omega \subset \R^n$ be a bounded Lipschitz domain with $n\geq 3$ and let us fix
    \begin{equation}\label{delta_0}
    	\delta_0 \in \left(-\frac{n}{2}, \frac{n}{2}-2\right).
    	\end{equation}
     Let $A(x)=\LC A^{jk}(x)\RC_{j,k=1}^n$ be a smooth real-valued positive definite symmetric matrix on $\R^n$ such that 
    \begin{equation}
        \label{matrix_iden}
        \LC A^{jk}(x)\RC_{j,k=1}^n =\mathds{1}_{n\times n} \quad \text{for $x\in \Omega_e=\R^n\setminus \overline\Omega$.}
    \end{equation}
    We consider the equation
	\begin{align}\label{equ: main}
		L_Au:=-\nabla \cdot (A(x)\nabla u)+P_0 u=0 \qquad &\text{ in }\Omega,
	\end{align}
	in the distributional sense, where the local anisotropic principal part has the divergence form
    
$$-\nabla \cdot (A(x)\nabla u)=-\sum_{j,k=1}^n \frac{\p}{\p x_j}\left(A^{jk}(x)\frac{\partial u}{\partial x_k} \right).$$   
and where $P_0$ is the fractional (variable coefficient) polyharmonic operator defined by 
	\begin{equation}\label{P_0_def}
		P_0u :=  \sum_{k=1}^N p_k\, (-\Delta)^{s_k}\left(p_k u\right) \qquad \text{for some $\{p_k\}_{k=1}^N\subset C^{\infty}_0(\R^n)$},
	\end{equation}
     for some $N\geq 1$, some $\{s_k\}_{k=1}^N\subset (0,\frac{1}{2}]$ with $s_1<\ldots<s_N$ and some collection of complex-valued functions $\{p_k\}_{k=1}^N\subset C^{\infty}_0(\R^n)$ that satisfy 
    \begin{equation}
        \label{p_k_def}
        p_k(x) \equiv b_k \quad \text{for $k=1,\ldots,N$ and all $x$ in some open neighbourhood $\mathcal U$ of $\overline{\Omega}$,}
    \end{equation}
    where $\{b_k\}_{k=1}^N\subset \C\setminus \{0\}$ is some set of $N$ nonzero numbers. We aim to study an inverse problem for solutions to \eqref{equ: main} subject to exterior measurements of its solutions. Precisely, our goal is to determine an a priori unknown matrix $A(x)$ in \eqref{equ: main} from the knowledge of the following (exterior) Cauchy dataset  
	$$ \mathcal C_A\subset W^{2,2}_{\delta_0}(\Omega_e) \times L^2_{\delta_0+2}(\Omega_e) \qquad \text{with}\qquad \Omega_e=\R^n\setminus \overline{\Omega},$$
	defined by
	\begin{equation}
		\label{eq: Cauchy data}
		\mathcal{C}_A:= \left\{\left(  u|_{\Omega_e},  \left.(L_Au) \right|_{\Omega_e} \right)\,:\, 
		\text{$u\in W^{2,2}_{\delta_0}(\R^n)$ \, with \, $L_Au=0$ \, in $\Omega$}
		\right\},
	\end{equation}
	where $\delta_0$ is as in \eqref{delta_0} and we recall that the operator $L_A$ is given by \eqref{equ: main} and that the equation is to be understood in the distribution sense. We will define the weighted Sobolev spaces $L^2_\delta(\R^n)$ and $W^{k,2}_\delta(\R^n)$ in Section \ref{sec: preliminary}. The richness of the Cauchy dataset, $\mathcal C_A$, will be discussed in Proposition~\ref{prop_solvability}, see also (iii) in Remark~\ref{rmk_thm_ip1}. Our first inverse problem can now be formulated as follows:
	
	\begin{enumerate}[\textbf{(IP1)}]
		\item\label{Q:IP1} \textbf{Inverse Problem 1.}  Can one uniquely determine the matrix-valued function $A$ in $\Omega$ from the exterior Cauchy dataset $\mathcal{C}_A$ defined by \eqref{eq: Cauchy data}?
	\end{enumerate}
	
	We prove the following global uniqueness result for \ref{Q:IP1} in Section~\ref{sec: IP}.

	\begin{theorem}[Global anisotropic uniqueness result]\label{Thm: global uniqueness 1}
		Let $\Omega \subset \R^n$, $n\geq 3$, be a bounded Lipschitz domain, let $N \in \N$, let $\{s_k\}_{k=1}^N\subset (0,\frac{1}{2}]$ with $s_1<\ldots<s_N$. Let $\{p_k\}_{k=1}^N\subset C^{\infty}_0(\R^n)$ satisfy \eqref{p_k_def}. For $j=1,2,$ let $A_j\in C^{\infty}(\R^n;\R^{n^2})$ be a positive definite symmetric matrix that satisfies \eqref{matrix_iden} and subsequently define $\mathcal{C}_{A_j}$ as the exterior Cauchy dataset \eqref{eq: Cauchy data} (with $A=A_j$ and $\delta_0$ fixed as in \eqref{delta_0}). Then,
		\begin{align}
			\mathcal{C}_{A_1}=\mathcal{C}_{A_2}\quad \text{implies that}\quad A_1=A_2 \text{ in }\Omega.
		\end{align}
	\end{theorem}
    \begin{remark}
        \label{rmk_thm_ip1}
        Let us make some remarks about the above theorem.
        \begin{itemize}
            \item[(i)]{To the best of our knowledge, the above theorem is new even in the case $N=1$. Let us also point out that in the case that $N=1$, the condition \eqref{p_k_def} is not needed as long as $p_1$ does not vanish at any point in an open neighborhood of $\overline{\Omega}$. Observe that in comparison with the anisotropic Calder\'{o}n problem, and somewhat surprisingly, we have proven that in the presence of the nonlocal lower order terms $P_0u$ in the equation $L_Au=0$ in $\Omega$, there is no diffeomorphism gauge and one indeed recovers the matrix $A$ uniquely from the exterior Cauchy dataset.}

            \item[(ii)]{The assumption on the dimension, namely $n\geq 3$, is for simplicity of presentation of the forward problem and is related to the study of Fredholm properties of the operator $u\mapsto \nabla\cdot(A(x)\nabla u)$ which itself depends in a key way on mapping properties of the Laplacian, see example \cite{McOwen1979TheBO, Bartnik1986TheMO}. In dimensions three and higher, the Laplacian acts as an isomorphism between certain pairs of weighted Sobolev spaces of the form $W^{k,2}_{\delta}(\R^n)$. This will no loner be true in dimension two and one needs to work with more complicated Sobolev spaces.}
            
            \item[(iii)]{Both the assumptions $\{s_k\}_{k=1}^{N}\subset (0,\frac{1}{2}]$ and that $\{p_k\}_{k=1}^N\subset C^{\infty}_0(\R^n)$ are also related to the forward theory and the structure of Cauchy dataset $\mathcal C_A$. Under the previous two assumptions, we have the Fredholmness properties for the operator $L_A$ that is proven in Proposition~\ref{prop_solvability}. Finally, let us mention that, as we will see in the proof of Theorem~\ref{Thm: global uniqueness 1}, it suffices for us to work with certain smooth elements in $\mathcal C_A$ subject to certain growth conditions at infinity. Having this in mind, let us note that we could have alternatively defined a broader Cauchy dataset as follows
            	$$  \widetilde{\mathcal C}_{A} = \{ \left(  u|_{\Omega_e},  (L_Au)|_{\Omega_e}\right)\,:\, 
            	\text{$u\in C^{\infty}(\R^n)$ \, with \, $L_Au=0$ \, in $\Omega$}\},$$
            	where we recall that $L_A:C^{\infty}(\R^n) \to C^{\infty}(\R^n)$ since $\{p_k\}_{k=1}^N \subset C^{\infty}_0(\R^n)$. Thus, our formulation of the Cauchy dataset $\mathcal C_A$ should be viewed as working with less data compared to $\widetilde{\mathcal C}_A$.
        }

            \end{itemize}
    \end{remark}

    \subsection{Recovery of a zeroth order local term}
	Next, we discuss a partial data inverse problem for recovering zeroth order coefficients for fractional polyharmonic equations. We assume the more restrictive condition that $\{ b_k \}_{k=1}^N\subset (0,\infty)$ and an additional condition on the zeroth order coefficient.

	To state this partial data inverse problem, let us again consider $\Omega \subset \R^n$ to be a bounded Lipschitz domain with $n\geq 2$, let $\{s_k\}_{k=1}^N$ be such that \ref{exponent condition} is satisfied. 
	Next, let $q\in L^{\infty}(\Omega)$ and consider the exterior Dirichlet value problem
	\begin{equation}\label{equ: main 2}
		\begin{cases}
			P_q u =0 &\text{ in }\Omega, \\
			u=f &\text{ in }\Omega_e,
		\end{cases}
	\end{equation}
    where 
    \begin{equation}\label{P_def}
		P_qu :=  \sum_{k=1}^N b_k (-\Delta)^{s_k}u + q(x)u , 
	\end{equation}
	We will assume that $b_k>0$ for all $k=1,\ldots, N$, and that
	\begin{equation}\label{eigenvalue condition}
		\text{$0$ is not a Dirichlet eigenvalue of $P_q$}
	\end{equation}
	in the sense that 
	\begin{equation}
		\begin{cases}
			P_q u =0 & \text{ in }\Omega \\
			u=0 &\text{ in }\Omega_e
		\end{cases}
		\text{ implies }u\equiv 0 \text{ in }\R^n.
	\end{equation}
	Letting $W_1, W_2\Subset \Omega_e$ be two bounded nonempty open sets, and by using the condition \eqref{eigenvalue condition}, we can define the so-called \emph{Dirichlet-to-Neumann} (DN) map of \eqref{equ: main 2} via
	\begin{equation}\label{DN map}
		\Lambda_q: \wt H^{s_N}(W_1) \to H^{-s_N}(W_2), \quad f\mapsto \left. (P_q u_f) \right|_{W_2},
	\end{equation}
	where $u_f \in H^{s_N}(\R^n)$ is the unique solution to \eqref{equ: main 2}. We refer the reader to Section~\ref{sec: preliminary: fcn} for the definition of the $\wt H^{s}(U)$ spaces (for $s\in \R$) and to Section~\ref{sec: exterior problem} for the well-posedness of equation \eqref{equ: main 2}.
	
	\begin{enumerate}[\textbf{(IP2)}]
		\item\label{Q:IP2} \textbf{Inverse Problem 2.}  Can one uniquely determine the potential $q$ in $\Omega$ from the exterior DN map $\Lambda_q$ defined by \eqref{DN map}?
	\end{enumerate}
	
	We prove the following uniqueness result for \ref{Q:IP2} in Section~\ref{sec: IP}.
	
	\begin{theorem}[Global uniqueness for bounded potentials]\label{Thm: global uniqueness 2}
		Let $\Omega \subset \R^n$, $n\geq 2$, be a bounded Lipschitz domain, and let $W_1, W_2\Subset \Omega_e$ be nonempty bounded open sets. Let $N \in \N$, $\{b_k\}_{k=1}^N \subset (0,\infty)$ and let $\{s_k\}_{k=1}^N$ satisfy \ref{exponent condition}. For each $j=1,2,$ let $q_j\in L^\infty(\Omega)$ satisfy \eqref{eigenvalue condition} and define $\Lambda_{q_j}$ to be the DN map \eqref{DN map} (with $q=q_j$). Then, 
		\begin{align}
		\Lambda_{q_1}f=\Lambda_{q_2}f \text{ for any }f\in C^\infty_0(W_1), \quad \text{implies that}\quad q_1=q_2 \text{ in }\Omega.
		\end{align}
	\end{theorem}

	\begin{remark}
    Note that when $N=1$, and the exponent $s_1$ belongs to $(0,1)$, the preceding theorem reduces to the known main result of  \cite{GSU20}. Let us also comment that by rewriting the nonlocal Schr\"odinger equation $P_q$ as 
		$$
		P_q := \psi(-\Delta) +q(x),
		$$
		where $\psi (\lambda)=\sum_{k=1}^N b_k \lambda^{s_k}$, we obtain a formulation of the inverse problem via the \emph{Bernstein function}\footnote{A function $f:(0,\infty) \to \R$ is a Bernstein function if $f\in C^\infty((0,\infty))$, $f(\lambda)\geq 0 $ for all $\lambda>0$, and $(-1)^k \frac{d^kf (\lambda)}{d\lambda^k} \geq 0$ for all $\lambda>0$ and for all $k\in \N$. A typical example of Bernstein functions is $f(\lambda)=\lambda^s$, for any $s\in (0,1)$. It is also known that $b_1f_1+b_2f_2$ is a Bernstein function for any constants $b_1,b_2>0$, provided $f_1$ and $f_2$ are Bernstein functions.}. Hence, Theorem \ref{Thm: global uniqueness 2} can be regarded as a generalization of the fractional Calder\'on type inverse problem associated with Bernstein-type operators which, to the best of our knowledge, has not been solved before. We refer the reader to \cite{KM18_Bernstein} for related studies for extension problems of complete Bernstein functions associated with the Laplace operator. Another related study was investigated in \cite{LLL_poly}.
	\end{remark}

	Theorem \ref{Thm: global uniqueness 2} can be proved by using the following Runge approximation property for solutions of \eqref{equ: main 2}, which itself involves the entanglement principle of Theorem \ref{thm: ent}. The Runge approximation property may be of independent interest in control theory. We prove this theorem in Section~\ref{sec: IP}.
	
	\begin{theorem}[Runge approximation]\label{Thm: Runge}
		Let $\Omega\subset \R^n$ be a bounded open set, and let $W\Subset \Omega_e$ be a bounded nonempty open set. Let $N \in \N$, $\{b_k\}_{k=1}^N \subset (0,\infty)$ and let $\{s_k\}_{k=1}^N$ satisfy \ref{exponent condition}. Let $q\in L^\infty(\Omega)$ satisfy \eqref{eigenvalue condition}. Then, given any function $g\in L^2(\Omega)$ and any $\eps>0$, there exists a solution $u=u_\eps\in H^{s_N}(\R^n)$ of equation \eqref{equ: main 2} for some exterior Dirichlet data $f=f_\epsilon\in C^{\infty}_0(W)$ such that 
		$
		\left\| u_\eps -g \right\|_{L^2(\Omega)}<\eps.$
	\end{theorem}

    Apart from Theorems \ref{Thm: global uniqueness 1} and \ref{Thm: global uniqueness 2}, we would expect further applications of our entanglement principle in the study of inverse problems for systems of nonlocal equations. We leave this as a future direction of research.
	
	\subsection{Outline of the key ideas in the proof of Theorem~\ref{thm: ent}} \label{sec_outline_proof}
	
	Let us discuss some of the key ideas in the proof of Theorem~\ref{thm: ent} in Section~\ref{sec: entanglement}. Our starting point will be to show that the principle can be derived from an analogous statement for smooth functions with super-exponential decay at infinity, see the statement of Theorem~\ref{thm_ent_smooth} for this version of the theorem. This is not surprising as fractional Laplace operators commute with convolution operators on $\mathbb R^n$, as can be readily seen from their definition via Fourier transforms. The proof of the smooth case $\{u_k\}_{k=1}^N\subset C^{\infty}(\R^n)$  will then be divided into three main steps. \\
	
	\noindent {\bf Step I.} In the first step, we proceed to make a hidden connection between the analogue of equation \eqref{ent_u_cond} in the smooth case (see Theorem~\ref{thm_ent_smooth}) and the holomorphic function 
	$$F:\{z\in \C\,:\, \mathrm{Re}(z)\geq 0\} \to \C,$$ 
	that is (for now formally) defined via the expression
	\begin{equation}\label{F_intro} 
		F(z):=  \sum_{k=1}^N \frac{\Gamma(z+1+\alpha_k)}{\Gamma(-\alpha_k)\Gamma(1+\alpha_k)}\int_0^\infty (e^{t\Delta} \,u_k)(x) t^{-(z+1+\alpha_k)}\, dt\end{equation}
	where $\alpha_k\in (0,1)$, $k=1,\ldots,N$ is the fractional part of $s_k$ and $x$ is a fixed point inside $\mathcal O$. We will prove that \eqref{ent_u_cond} implies that the function $F(z)$ above must vanish on positive integers, that is to say, 
	\begin{equation}\label{F_m_intro} F(m)=0 \quad \text{for $m=1,2,\ldots$}.\end{equation}
	We will carefully analyze the well-posedness of the definition \eqref{F_intro} showing that it is indeed a holomorphic function of $z$ in the right half-plane. We will then derive precise bounds for its growth rates as $|z|\to \infty$, see Lemma~\ref{lem_analytic_bounds}. The remaining part of this first step is to establish Proposition~\ref{prop_F_zero}, showing that the only holomorphic function on $\left\{ z\in \C: \, \textrm{Re}(z)\geq 0\right\}$ that vanishes on positive integers and enjoys the growth rates of Lemma~\ref{lem_analytic_bounds} is the zero function. This part relies crucially on an interpolation theorem for holomorphic functions in the same spirit as Carlson's theorem in complex analysis. The version that we need here is due to Pila \cite{pila_05}, see Theorem~\ref{thm_Pila} for its statement. \\
	
	\noindent{{\bf Step II.}} Once we have established that $F(z)=0$ for all $z\in \{z\in \C\,:\,\textrm{Re}(z)\geq 0\}$, we aim to see what further information about the functions $u_k$, $k=1,\ldots,N$ may be inferred from it. Let us also comment that this is a key step that diverges from the approach in \cite{FKU24}. In \cite{FKU24} the authors showed that an analogous expression as $F(z)$ above appears on closed Riemannian manifolds. Subsequently, they showed that $F(z)$ in their setup is globally holomorphic away from nonpositive integers, thanks in part to the large time exponential decay of the heat semigroup $e^{t\Delta_g}$ when acting on $\Delta_g u$ with $u$ smooth (the operator $\Delta_g$ denotes the Laplace-Beltrami operator). This allowed them to perform singularity analysis near the poles of $F(z)$ and conclude that each of the functions $u_k$ must be zero in the case of closed Riemannian manifolds, see also \cite[Remark 1.9]{FKU24}.
	
	However, as it can be readily seen from the expression \eqref{heat_kernel} for the heat kernel on $\R^n$, the large time behaviour of the heat semigroup in $\R^n$ only has a polynomial decay of order $t^{-\frac{n}{2}}$ and thus the expression $F(z)$ will become divergent as one moves into the left half-plane $\textrm{Re}(z)\leq 0$. Nevertheless, we will prove that under Schwartz class decay for the functions $u_k$, it is possible to meromorphically continue the function $F(z)$ into the left half-plane. This is analogous to the well-known meromorphic continuation of the Gamma function to the left half-plane based on a recursive equation that it enjoys in the right half-plane, namely \eqref{recursion of Gamma function}. We refer the reader to Lemma~\ref{lem_meromorphic} for the expression of this meromorphic extension.

    Once this extension is obtained, we proceed to perform singularity analysis near its poles and show that this leads to disentanglement of the terms in the expression \eqref{ent_u_cond}. Assuming the condition \ref{exponent condition}, this leads us to show that the following specific moments must vanish for each fixed $x\in \mathcal O$ and each $m\in \N\cup \{0\}$,
	\begin{equation}\label{moments_intro} 
		\begin{cases}
			\int_{\R^n} u_k(y)\, |x-y|^{2m}\,dy =0,  &\text{if $n$ is even,}\\
			\int_{\R^n} u_k(y)\, |x-y|^{2m+1}\,dy =0,  &\text{if $n$ is odd}.
		\end{cases}
	\end{equation}
	We remark that in the case of odd dimension and if two exponents differ by an odd multiple of $\frac{1}{2}$, the singularity analysis becomes more complicated as some pairs of terms in \eqref{ent_u_cond} create a resonance effect leading to more complicated expressions. For the sake of clarity of our presentation, we decided to remove this possibility by the extra assumption in odd dimensions. We comment that up to the end of this step, only Schwartz class decay is needed.\\
	
	\noindent{{\bf Step III.}} The last step of our analysis is to show that the vanishing of the previous moments for each $x\in \mathcal O$ would imply that the functions $u_k$ must all vanish globally on $\R^n$. This is the only step of the proof where we need to impose more spatial decay on the functions $u_k$ than Schwartz class decay. Indeed, it seems we must have super-exponential decay to be able to conclude this step. Our proof relies on showing that under super-exponential decay, the vanishing of the previous moments implies that the spherical averages of $u_k$ must be zero over all spheres centered at the set $\mathcal O$. This step is based on the study of Fourier--Laplace transforms. The proof is then completed thanks to well-known support theorems for these geometrical Radon type transforms.
	
	\subsection{Organization of the paper}
	The paper is organized as follows. In Section \ref{sec: preliminary}, we introduce basic notions used in this work and prove the solvability and well-posedness for \eqref{equ: main} and \eqref{equ: main 2}, respectively. In Section \ref{sec: entanglement}, we derive the entanglement principle, involving several tools including analytic interpolation as well as reduction to spherical mean transforms. In Section~\ref{sec: IP}, we apply the entanglement principle to show the global uniqueness results for \ref{Q:IP1}--\ref{Q:IP2} as well as proving the Runge approximation property for solutions to \eqref{equ: main_solv}.
	
	\section{Preliminaries}\label{sec: preliminary}
\subsection{Function spaces}
	\label{sec: preliminary: fcn}
    We briefly discuss the notations for the weighted Sobolev spaces as well as the notation of fractional Sobolev spaces.

    \subsubsection{Weighted Sobolev spaces}
Following the notations in \cite{4b44167c-d575-3418-b803-b8adaeefcb60, McOwen1979TheBO}, we define for $\delta\in \R$ the weighted Sobolev space $L^2_\delta(\R^n)$ as the space of measurable functions $u\in L^2_{\textrm{loc}}(\R^n)$ such that the norm
\begin{equation}\label{def_weight_sobolev}
\|u\|_{L^2_\delta(\R^n)}= \left(\int_{\R^n} (1+|x|^2)^{\delta}\,|u(x)|^2\,dx \right)^{\frac{1}{2}}
\end{equation}
is finite. Next, given any $k=0,1,\ldots$ and any $\delta\in \R$ we define $W^{k,2}_\delta(\R^n)$ in an analogous way corresponding to the norms
\begin{equation}\label{def_weight_sobolev_derivative}
\|u\|_{W^{k,2}_\delta(\R^n)}= \sum_{j=0}^k \sum_{|\beta|=j}\|D^{\beta}u\|_{L^2_{\delta+j}(\R^n)},
\end{equation}
where the second summation is taken over multi-indexes $\beta\in (\N\cup\{0\})^n$ with $|\beta|=\beta_1+\ldots+\beta_n=j$ and we have $D^\beta = \frac{\p^{|\beta|}}{\p x_1^{\beta_1}\ldots\, \p x_n^{\beta_n}}.$ The spaces $W^{k,2}_\delta(U)$ are to be understood similarly for any open set $U\subset \R^n$. Let us also mention that the notations $W^{k,2}(\R^n)=H^{k}(\R^n)$ for $k\in \N$ and the notation $W^{0,2}(\R^n)=L^2(\R^n)$ will be reserved for the standard Sobolev spaces, not to be confused with the weighted Sobolev spaces above. 



The above weighted Sobolev spaces will be key when it comes to discussing the structure of the Cauchy dataset $\mathcal C_A$. In particular, we will need to use the following elliptic regularity result of \cite{Bartnik1986TheMO} that we recall here. We caution the reader that the notations of Bartnik for weighted Sobolev spaces are slightly different from the standard notation \eqref{def_weight_sobolev}-\eqref{def_weight_sobolev_derivative}. Thus, for the sake of the reader's convenience, we will state the lemma here, adjusted to fit our notation.

\begin{lemma}(Elliptic regularity, cf. Proposition 1.6 in \cite{Bartnik1986TheMO})
	\label{lem_elliptic_regularity_weighted}
	Let $n\geq 2$ and let $A$ be a real-valued positive definite symmetric matrix that satisfies \eqref{matrix_iden}. Let $\mathcal L_0:= \nabla\cdot(A(x)\nabla)$ and finally let $\delta\in \R$.  If $u\in L^{2}_{\delta}(\R^n)$ and $\mathcal L_0 u \in L^{2}_{\delta+2}(\R^n)$, then $u \in W^{2,2}_{\delta}(\R^n)$ and there holds
	$$ \|u\|_{W^{2,2}_{\delta}(\R^n)} \leq C \left(  \|\mathcal L_0 u\|_{L^{2}_{\delta+2}(\R^n)} + \|u\|_{L^{2}_{\delta}(\R^n)}  \right),$$
	for some $C>0$ independent of $u$.
\end{lemma}

\subsubsection{Fractional Sobolev spaces} 
	Recall that the fractional Laplacian $(-\Delta)^s$, $s\geq 0$, is given via the Fourier transform as follows.
	\begin{equation}
		(-\Delta)^s u=\mathcal{F}^{-1}\left\{ \abs{\xi}^{2s}\mathcal{F}u(\xi)\right\}, \quad \text{for }u\in \mathcal{S}(\R^n),
	\end{equation}
	where $\mathcal{F}$ and $\mathcal{F}^{-1}$ denote the Fourier and inverse Fourier transform, respectively. For the functional spaces, we write $H^s (\R^n ) = W^{s,2}(\R^n)$ for the $L^2$-based Sobolev space with norm
	\begin{equation}
		\norm{u}_{H^s(\R^n)} = \left\| \langle D \rangle^s \mathcal F u\right\|_{L^2(\R^n)},
	\end{equation}
	for any $s\in \R$, where $\langle \xi \rangle =\LC1+\abs{\xi}^2\RC^{1/2}$. Given a nonempty open set $U\subset \R^{n}$, the space $C_0^{\infty}(U)$ consists of all $C^{\infty}(\mathbb{R}^{n})$-smooth functions with compact support in $U$. Analogously as in \cite{GSU20}, we adopt for each $s\in \R$, the notations 
	\begin{align*}
		H^{s}(U) & :=\left\{u|_{U}: \, u\in H^{s}(\R^{n})\right\},\\
		\wt H^{s}(U) & :=\text{closure of \ensuremath{C_{0}^{\infty}(U)} in \ensuremath{H^{s}(\R^{n})}},\\
		H_{0}^{s}(U) & :=\text{closure of \ensuremath{C_{0}^{\infty}(U)} in \ensuremath{H^{s}(U)}}.
	\end{align*}
	The space $H^{s}(U)$ is complete in the sense that
	\[
	\|u\|_{H^{s}(U)}:=\inf\left\{ \|v\|_{H^{s}(\mathbb{R}^{n})}: \, v\in H^{s}(\mathbb{R}^{n})\mbox{ and }v|_{U}=u\right\} .
	\]
	The space $H^{-s}(U)$, with any $s\in \R$, may be viewed as the topological dual space of $\wt H^s(U)$.
	We also we use the notation $\langle v, w\rangle_{H^{-s}(U),\widetilde H^s(U)},$ to denote the continuous extension of 
	$$ (v,w)_{L^2(U)}= \int_{U} v(x)\,\overline{w}(x)\,dx \qquad \forall\, v,w \in C^{\infty}_0(U),$$
	as a sesquilinear form to all of $H^{-s}(U) \times \widetilde H^{s}(U).$ We also recall a mapping property for the fractional Laplacian.
	\begin{lemma}[Lemma 2.1 in \cite{GSU20}]\label{Lem: mapping prop of frac Lap}
		For $s\geq 0$, the fractional Laplacian extends as a bounded map
		\begin{equation}
			(-\Delta)^s : H^a(\R^n)\to H^{a-2s}(\R^n), \text{ for }a\in \R.
		\end{equation}
	\end{lemma}
	
	\subsection{Nonlocal operators defined via the heat semigroup}
	
	For values $s\in (0,1)$, there are several equivalent definitions for fractional Laplace operator $(-\Delta)^s$. Here, let us use the heat semigroup definition which will be suitable for our analysis. We first recall the definition of the heat semigroup. Let
	\begin{align}\label{heat_kernel}
		p_t(y):=\frac{1}{(4\pi t)^{n/2}}e^{-\frac{|y|^2}{4t}}, \text{ for }y\in \R^n \text{ and }t>0, 
	\end{align}
	be the heat kernel of the heat operator $\p _t -\Delta$ for $(t,x)\in (0,\infty)\times \R^n$. 
Let $u\in H^s(\R^n)$, and define $e^{t\Delta}u\colon [0,\infty)\times \R^n\to\C$ by
	\begin{align}\label{heat kernel representation}
		e^{t\Delta}u(x):=\int_{\R^n} p_t(x-y)u(y)\, dy\quad \text{for}\quad t>0
	\end{align}
	and also $e^{0\Delta}u(x):=u(x)$. Then $	e^{t\Delta}u \in C^{\infty}([0,\infty);H^s(\R^n))$ is the unique solution to 
	\begin{align}\label{heat equation}
		\begin{cases}
			\LC \p_t -\Delta \RC  U=0 &\text{ in }(0,\infty) \times \R^n,\\
			U(x,0)=u(x)&\text{ in }\R^n.
		\end{cases}
	\end{align}
	It is well known that
	\begin{align}
		\left\|	e^{t\Delta}u\right\|_{H^s(\R^n)}\leq \norm{u}_{H^s(\R^n)}, \text{ for }t\geq 0.
	\end{align}

	Next, for specific values $s\in (0,1)$, we recall the well-known equivalent expression for the fractional Laplacian given via
	\begin{equation}\label{fractional heat semi}
		\LC -\Delta \RC ^s u(x) =\frac{1}{\Gamma (-s)}\int_0^\infty \frac{e^{t\Delta} u(x)-u(x)}{t^{1+s}}\, dt,
	\end{equation}
	where $\Gamma(\cdot)$ is the Gamma function defined by 
	\begin{equation}\label{Gamma function}
		\Gamma(z)=\int_0^\infty e^{-t}t^{z-1}\, dt, \quad \text{for $\textrm{Re}(z)>0$.} 
	\end{equation}
	As the Gamma function plays an essential role in our analysis, let us also mention the recursion formula 
	\begin{equation}\label{recursion of Gamma function}
		\Gamma(z+1)=z\,\Gamma(z)\quad \text{or}\quad \Gamma(z)=\frac{\Gamma(z+1)}{z} \quad \text{$\textrm{Re}(z)>0$}.
	\end{equation}
	Indeed, the above recursion formula is important as it allows one to extend the Gamma function as a holomorphic function to all of the complex plane $\C$ except at nonpositive integers where the extended function will have simple poles. Throughout the remainder of this paper, the Gamma function $\Gamma$  is to be understood in this extended sense. Let us close this section by noting that several other equivalent definitions of the fractional Laplace operator are known, see for example the survey article \cite{Kwansnicki17} for ten equivalent definitions.

	\subsection{On the Cauchy dataset $\mathcal C_A$}
	\label{sec: C_q}
	
    In order to prove Theorem~\ref{Thm: global uniqueness 1}, we first need to show that the Cauchy dataset $\mathcal C_A$ makes sense, that is to say, it is not empty and possesses enough elements for us to study the inverse problem.  In fact, as we will see in a moment, this set is infinite-dimensional and we can precisely categorize many of its elements, sufficient for our purposes of solving \ref{Q:IP1}. To this end, let us discuss the Poisson equation 
	\begin{equation}\label{equ: main_solv}
		  L_A u =F \qquad \text{ in }\R^n,
	\end{equation}
	with  $F\in L^2_{\delta_0+2}(\R^n)$, where $\delta_0$ is as in \eqref{delta_0} and is fixed throughout this manuscript and we refer the reader to Section~\ref{sec: preliminary: fcn} for the definition of weighted Sobolev spaces. Later, when it comes to discussing the Cauchy dataset $\mathcal C_A$, we will naturally impose that $\supp F \subset \overline{\Omega_e}$ where we recall that $\Omega_e=\R^n\setminus \overline\Omega$. To solve the inverse problem \ref{Q:IP1}, it suffices for us to only work with $F \in C^{\infty}_0(\Omega_e)$.  
	
	\begin{remark}
		\label{remark_fractional_domain_weight}
		Let us make the remark that given any $\delta\in \R$, the operator $L_A: W^{2,2}_\delta(\R^n) \to L^2_{\delta+2}(\R^n)$ defined by \eqref{equ: main} is a bounded linear operator. Recall that $L_A$ has a local divergence part for which this mapping property is obvious. Considering next the fractional polyharmonic part given by  $P_0$ given by \eqref{P_0_def}, we note that given any $u\in W^{2,2}_\delta(\R^n)$, we have that $p_k u \in H^2(\R^n)$ for $k=1,\ldots,N$, as the functions $p_k$ are all smooth and with compact support. Thus, using again the fact that $p_k$'s are compactly supported together with the mapping property for fractional Laplace operators in Lemma~\ref{Lem: mapping prop of frac Lap}, we deduce that $p_k (-\Delta)^{s_k} (p_ku) \in L^2_{\delta'}(\R^n)$ for any $\delta'\in \R$. Recall that $\{s_k\}_{k=1}^N \subset (0,\frac{1}{2}].$ 
	\end{remark}

	In the following proposition, it is useful to note that given any $\delta\in \R$, the topological dual of $L^2_{\delta}(\R^n)$ is $L^2_{-\delta}(\R^n)$ and also that the adjoint of $L_A:W^{2,2}_{\delta}(\R^n)\to L^2_{\delta+2}(\R^n)$ is given by $\overline{L_A}: L^{2}_{-\delta-2}(\R^n) \to (W^{2,2}_{\delta}(\R^n))^\star$. Note also that for $n\geq3$, we have for any $\delta\in \R$,
	\begin{equation}\label{delta_0_property}
	 \delta \in \left( -\frac{n}{2}, \frac{n}{2}-2\right)  \iff   -\delta-2 \in \left( -\frac{n}{2}, \frac{n}{2}-2\right).\end{equation}
    
	\begin{proposition}[Solvability of Poisson equation]\label{prop_solvability}
	Let $n\geq 3$. Let $A$ be a smooth positive definite real-valued symmetric matrix on $\R^n$ that satisfies \eqref{matrix_iden}. Let $N \in \N$, let $\{p_k\}_{k=1}^N \subset C^\infty_0(\R^n)$ satisfy \eqref{p_k_def} and finally let $\{s_k\}_{k=1}^N \subset (0,\frac{1}{2}]$.  Let $\delta_0$ be as in \eqref{delta_0}. The operator
		\begin{equation}\label{mapping solvable}
			L_A:  W^{2,2}_{\delta_0}(\R^n) \to L^2_{\delta_0+2}(\R^n)
		\end{equation}
		defined in \eqref{equ: main} is Fredholm with index zero. Defining
		\begin{align}\label{K_A_def}
			\mathcal{K}_{A}:=\left\{ u\in W^{2,2}_{\delta_0}(\R^n): \, L_A u=0 \text{ in }\R^n\right\},
		\end{align}
	we have the inclusions
	\begin{equation}\label{K_A_reg}
		\mathcal K_A \subset W^{k,2}_{-2-\delta_0}(\R^n) \cap W^{k,2}_{\delta_0}(\R^n) \quad \text{for any $k\in \{0\}\cup \N$.}
	\end{equation} 
		There are two mutually exclusive possibilities: 
		\begin{enumerate}[(i)]
			\item $\mathcal{K}_A=\{0\}$. In this case, for $F \in L^2_{\delta_0+2}(\R^n)$, the Poisson equation \eqref{equ: main_solv} has a unique solution $u\in W^{2,2}_{\delta_0}(\R^n)$.
			
			\item $\dim(\mathcal{K}_A)=m$, for some $m\in \N$. In this case, for $F \in L^2_{\delta_0+2}(\R^n)$, the Poisson equation \eqref{equ: main_solv} admits some solution $u\in W^{2,2}_{\delta_0}(\R^n)$ if and only if 
            \begin{equation}\label{phi_ortho} \left(F, \overline{\zeta}\right)_{L^{2}(\R^n)}=0 \qquad \forall\, \zeta \in \mathcal K_A.\end{equation}
		\end{enumerate}
       	 Moreover, if $F\in C^{\infty}_0(\R^n)$ satisfies \eqref{phi_ortho} then any solution $u\in W^{2,2}_{\delta_0}(\R^n)$ to the Poisson equation \eqref{equ: main_solv} will also be in $W^{k,2}_{\delta_0}(\R^n)\cap W^{k,2}_{-\delta_0-2}(\R^n)$ for all $k\in \N$ and thus in particular globally smooth in $\R^n$.
	\end{proposition}
\begin{remark}
    Let us remark that the appearance of the conjugate $\overline{\zeta}$ in \eqref{phi_ortho} above is because we are using sesquilinear forms and that the (formal) adjoint of $L_A$ is $\overline{L_A}$. Note also that in case (ii) above, if $u\in W^{2,2}_{\delta_0}(\R^n)$ is a solution to the equation \eqref{equ: main_solv}, then any other solution to the Poisson equation must then be of the form $u+\zeta$ for some $\zeta\in \mathcal K_A$.
    
 \end{remark}
 
 \begin{remark}
 		In the above proposition, it is crucial for us that $\delta_0$ is as in \eqref{delta_0}, that $\{s_k\}_{k=1}^N \subset (0,\frac{1}{2}]$, that $A$ satisfies \eqref{matrix_iden} and that $\{p_k\}_{k=1}^N \subset C^{\infty}_0(\R^n)$, so that the $p_k$'s have compact support. The assumption on $\delta_0$ provides suitable mapping properties for the local divergence part in the equation while the other assumptions simplify the treatment of the nonlocal perturbation and let us avoid discussing elliptic regularity properties in fractional weighted Sobolev spaces. We believe that some sets of assumptions need to be imposed here to have a well-posedness theory.

\end{remark}    

	\begin{proof}[Proof of Proposition \ref{prop_solvability}]
        Applying \cite[Theorem 0]{McOwen1979TheBO} we note that the mapping
        \begin{equation}
        	\label{Delta_weight_map}
        	\Delta: W^{2,2}_{\delta_0}(\R^n)\to L^{2}_{\delta_0+2}(\R^n)
        \end{equation}
        is an isomorphism. Using this result together with the fact that $A$ satisfies \eqref{matrix_iden}, we can apply \cite[Proposition 1.15]{Bartnik1986TheMO} to deduce that the operator 
        $$ \mathcal L_0: W^{2,2}_{\delta_0}(\R^n) \to L^2_{\delta_0+2}(\R^n),$$
        defined via 
        $$\mathcal L_0(u) := -\nabla \cdot\left( A(x)\nabla u\right),$$
        is also an isomorphism. We caution the reader that our notations for weighted Sobolev spaces follow the standard convention and are a bit different from that of Bartnik (up to a shift and sign change for the weight $\delta$). Note also that our choice of $\delta_0$ is non-exceptional in his framework which is crucial. We write $\mathcal L_0^{-1}:L^2_{\delta_0+2}(\R^n)\to W^{2,2}_{\delta_0}(\R^n)$ for its inverse. Next, by noting that $\{p_k\}_{k=1}^N\subset C^{\infty}_0(\R^n)$ and using the mapping property given in Lemma~\ref{Lem: mapping prop of frac Lap} together with the fact that $\{s_k\}_{k=1}^N\subset (0,\frac{1}{2}]$ it is straightforward to see that the mapping
        $$ p_k\,(-\Delta)^s p_k\,\mathcal L_0^{-1}: L^{2}_{\delta_0+2}(\R^n) \to H^{1}_0(U),$$
        is continuous for every $s\in (0,\frac{1}{2}]$ where $U$ is a nonempty bounded open set that contains the support of all the $p_k$'s. Together with the fact that $H^{1}_0(U)$ is compactly embedded in $L^2(U)$ for every nonempty bounded open set $U\subset \R^n$ and that $p_k$'s are compactly supported there, we deduce that the mapping
        $$ p_k(-\Delta)^{s_k} p_k\,\mathcal L_0^{-1}: L^{2}_{\delta_0+2}(\R^n) \to L^2_{\delta_0+2}(\R^n),$$
        is compact for every $k=1,\ldots,N$. This proves that the operator $L_A$ is indeed Fredholm with index zero. That the Fredholm alternative holds in the way that it is written is due to the regularity properties that we describe next.
        
		We only provide a sketch of the proof here as the details are classical techniques used in elliptic regularity. We begin with the claim in \eqref{K_A_reg} that $\mathcal K_A \subset W^{k,2}_{\delta_0}(\R^n)$ for any $k\in \N$. The claim is trivial if $m=0$. For $m\geq 1$, this follows from bootstrapping the boundedness of the maps (see Lemma~\ref{Lem: mapping prop of frac Lap}) 
        $$ (-\Delta)^{s_k} : H^{s}(\R^n) \to H^{s-2s_k}(\R^n)\subset H^{s-1}(\R^n) \quad k=1,\ldots,N,$$
        with the elliptic regularity property given by Lemma~\ref{lem_elliptic_regularity_weighted} (cf. \cite[Proposition 1.6]{Bartnik1986TheMO}), noting crucially that $\{s_k\}_{k=1}^N\subset (0,\frac{1}{2}]$ and that $p_k$'s are compactly supported. (In particular, since $\{p_k\}_{k=1}^N\subset C^{\infty}_0(\R^n)$, the operator of multiplication by $p_k$ is continuous from $H^l(\R^n)$ to  $W^{l,2}_{\delta'}(\R^n)$ 
        for $k=1,\ldots,N$, any $\delta'\in \R$ and $l\in \N \cup \{0\}.$)
        
        Let us now consider the second claim in \eqref{K_A_reg}. Assume again that $m\geq 1$ and consider any  nonzero element $\zeta \in \mathcal K_A$. There holds:
        $$ \mathcal L_0 \zeta = -P_0 \zeta \in L^{2}_{\delta'}(\R^n),$$
        for any $\delta'\in \R$, where we used the fact that $\{p_k\}_{k=1}^N\subset C^{\infty}_0(\R^n)$. Recalling that $\zeta \in W^{2,2}_{\delta_0}(\R^n)$ together with the fact that \eqref{delta_0_property} is satisfied with $\delta=\delta_0$, we may now apply \cite[Proposition 1.14]{Bartnik1986TheMO} to obtain the additional regularity property that $\zeta \in W^{2,2}_{-\delta_0+2}(\R^n)$. Analogously as in the previous paragraph, we can bootstrap this observation via Lemma~\ref{lem_elliptic_regularity_weighted} to show the claim. Finally, the claim regarding the extra regularity of solutions to \eqref{equ: main_solv} when $F\in C^{\infty}_0(\R^n)$ follows analogously.
	\end{proof}

    We will also need the following two lemmas for future reference in Section~\ref{sec: IP}. 
    
    \begin{lemma}
    	\label{lem_ortho}
    	Let $n\geq 3$. Assume that $\dim(\mathcal{K}_{A}) = m$, for some $m\in \N$, where $\mathcal{K}_A$ is given by \eqref{K_A_def}. Let $\zeta_1, \dots, \zeta_m \in \mathcal{K}_{A}$ be linearly independent functions.  Then, for any $c = (c_1, \dots, c_m) \in \mathbb{C}^m$, there exists $g \in C^\infty_0(\Omega_e)$ such that $\LC g, \overline{\zeta_l}\RC _{L^2(\R^n)} = c_l$ for $l = 1, \dots, m$. 
    	\end{lemma}
    
    	\begin{proof}
   		The proof is this lemma is similar to the proof of \cite[Lemma 3.9]{FKU24} with minor modifications. We include it for the sake of convenience for the reader. We shall show that the following linear map
    	\[
    	T: C_0^\infty(\Omega_e) \ni g \mapsto \big( \LC g, \overline{\zeta_1}\RC_{L^2(\R^n)}, \dots, \LC g, \overline{\zeta_m}\RC_{L^2(\R^n)}\big) \in \mathbb{C}^m
    	\]
    	is surjective. Suppose, for the sake of contradiction, that $T$ is not surjective. Then there exists a nonzero vector $a = (a_1, \dots, a_m) \in \mathbb{C}^m$ such that 
    	\begin{equation}
    	\label{eq_lemma_anal_1}
    	0 = T(g) \cdot a = \LC g, \overline{\zeta}\RC_{L^2(\R^n)},
    	\end{equation}
    	for all $g \in C^\infty_0(\Omega_e)$. Here, $\zeta := \sum_{l=1}^m a_l \zeta_l \in W^{2,2}_{\delta_0}(\R^n)$, and $\xi \cdot \eta = \sum_{l=1}^N \xi_l \overline{\eta_l}$ denotes the inner product of vectors $\xi, \eta \in \mathbb{C}^m$. It follows from \eqref{eq_lemma_anal_1} that $\zeta|_{\Omega_e} = 0$. Together with the fact that $\zeta \in \mathcal{K}_{A}$, we obtain from the expression for $L_A$ in \eqref{equ: main} that 
    	$$\zeta|_{\Omega_e}=0 \quad \text{and}\quad \sum_{k=1}^m p_k (-\Delta)^{s_k} (p_k \zeta) =0 \quad \text{in $\Omega_e$}$$
    	In particular, note that the first identity above also implies that the function $\zeta\in W^{2,2}_{\delta_0}(\R^n)$ satisfies super-exponential decay. Our entanglement principle Theorem~\ref{thm: ent} (applied with the choices $N=m$, $u_k=a_k\zeta$, $b_k=1$ for $k=1,\ldots,m$ and any nonempty bounded open set $\mathcal O\subset \Omega_e$) implies that $a_k\zeta = 0$ on $\R^n$ for all $k=1,\ldots,m$, and thus $a = 0$, which is a contradiction.
    	\end{proof}

 \begin{lemma}
            \label{lem_ortho_cond}
            Let $n\geq 3$. We adopt the same notations as in Proposition~\ref{prop_solvability}.  Suppose that $v\in L^{2}_{\delta_0}(\R^n)$ is a function such that the following conditional statement holds: $$\text{if}\quad \left(\textrm{$F\in C^{\infty}_0(\Omega_e)$ with $(F,\overline{\zeta})_{L^2(\Omega_e)}=0$ for all $\zeta\in \mathcal K_A$}\right) \quad \text{then}\quad (F,\overline{v})_{L^2(\Omega_e)}=0.$$ 
            Then, $v|_{\Omega_e}=\zeta|_{\Omega_e}$ for some  $\zeta\in \mathcal K_A$.
            \end{lemma}
        
            \begin{proof}
The argument that we present closely follows a part of the proof of \cite[Lemma 3.10]{FKU24} with some adjustments as we are using weighted Sobolev spaces on $\R^n$ and some care is needed. Let $\textrm{dim}\,\mathcal{K}_A=m\geq 1$ as the claim is trivial in the case $m=0$. Let $\mathcal  W:=(\sigma^{\delta_0\,}\mathcal K_A)|_{\Omega_e}$ where $\sigma=(1+|x|^2)^{\frac{1}{2}}$. Based on the proof of the previous lemma, we know that $\textrm{dim}\, \mathcal W=m$ and that if $\{\zeta_1,\ldots,\zeta_m\}$ is a basis for $\mathcal K_A$, then, 
\[
\mathcal W = \mathrm{span}\left\{(\sigma^{\delta_0}\zeta_1)|_{\Omega_e}, \dots, (\sigma^{\delta_0}\zeta_{m})|_{\Omega_e}\right\} \subset L^2(\Omega_e).
\]
Define $w:=(\sigma^{\delta_0}v)|_{\Omega_e} \in L^2(\Omega_e)$. Writing the orthogonal decomposition $L^2(\Omega_e) = \mathcal W \oplus \mathcal W^\perp$, we deduce that
\begin{equation}
\label{lem_density_new_1_2}
w = (\sigma^{\delta_0}\zeta)|_{\Omega_e} + w_0,
\end{equation}
where $\zeta \in \mathcal K_A$ and $w_0 \in \mathcal W^\perp$, i.e.,
\begin{equation}
\label{lem_density_new_2}
\big(w_0, \sigma^{\delta_0}\zeta_k\big)_{L^2(\Omega_e)} = 0 \quad \text{for all} \quad k = 1, \dots, m.
\end{equation}
We shall next show that the conditional statement in the lemma implies that $w_0 = 0$. To see this, let $\{h_\ell\}_{\ell=1}^\infty \subset C_0^\infty(\Omega_e)$ be such that
\begin{equation}
\label{lem_density_new_3}
\left\|h_\ell - \overline{w_0}\right\|_{L^2(\Omega_e)} \to 0 \quad \text{as} \quad \ell \to \infty.
\end{equation}
It follows from \eqref{lem_density_new_2} and \eqref{lem_density_new_3} that
\begin{equation}
\label{lem_density_new_4}
\lim_{\ell \to \infty} (h_\ell, \sigma^{\delta_0}\overline{\zeta_k})_{L^2(\Omega_e)} = 0 \quad \text{for all} \quad k = 1, \dots, m.
\end{equation}
By Lemma~\ref{lem_ortho}, there exist functions $\{\theta_k\}_{k=1}^N \subset C_0^\infty(\Omega_e)$ such that
\begin{equation}
\label{lem_density_new_5}
\left(\theta_k, \sigma^{\delta_0}\overline{\zeta_j}\right)_{L^2(\Omega_e)} = \delta_{kj} \quad \text{for all} \quad k, j = 1, \dots, m.
\end{equation}
Consider the sequence of functions $F_\ell \in C_0^\infty(\Omega_e)$ defined by
\[
F_\ell = h_\ell - \sum_{j=1}^{m} \left(h_\ell, \sigma^{\delta_0}\overline{\zeta_j}\right)_{L^2(\Omega_e)} \theta_j, \quad \ell = 1,2, \dots.
\]
It follows from \eqref{lem_density_new_5} that 
$$\left(\sigma^{\delta_0}F_\ell, \overline{\zeta_k}\right)_{L^2(\Omega_e)} =\left(F_\ell, \sigma^{\delta_0}\overline{\zeta_k}\right)_{L^2(\Omega_e)} = 0 \quad \text{for all $k = 1, \dots, m$, and $\ell=1,2,\dots$},$$ 
and therefore, by the hypothesis of the lemma, and the definition of $w$,
\begin{equation}
\label{lem_density_new_6}
\left(F_\ell, \overline{w}\right)_{L^2(\Omega_e)} = 0 \quad \text{for all} \quad \ell = 1, 2, \dots.
\end{equation}
We observe from \eqref{lem_density_new_3} and \eqref{lem_density_new_4} that
\begin{equation}
\label{lem_density_new_7}
\|F_\ell - \overline{w_0}\|_{L^2(\Omega_e)} \to 0 \quad \text{as} \quad \ell \to \infty.
\end{equation}
It follows from \eqref{lem_density_new_6} and \eqref{lem_density_new_7} that $\left(\overline{w_0}, \overline{w_0}\right)_{L^2(\Omega_e)} = 0$ and therefore, $w_0 = 0$, thus completing the proof of the lemma.
 \end{proof}

			\subsection{Well-posedness of the DN map with partial data}
			\label{sec: exterior problem}
			We will now discuss the well-posedness of the exterior-value problem \eqref{equ: main 2} under an additional constraint on $\{b_k\}_{k=1}^N$, namely that they are all positive real numbers.
			
			\begin{lemma}[Well-posedness]\label{Lem: well-posedness}
				Let $\Omega \subset \R^n$, $n\geq 2$, be a bounded Lipschitz domain, and let $W\Subset \Omega_e$ be a nonempty bounded open set. Let $N \in \N$, let $\{b_k\}_{k=1}^N \subset (0,\infty)$ and let $\{ s_k\}_{k=1}^N$ satisfy \ref{exponent condition}. Let $q\in L^\infty(\Omega)$ satisfy \eqref{eigenvalue condition}.  Then, given any $f\in C^\infty_0(W)$, there exists a unique solution $u\in H^{s_N}(\R^n)$ to \eqref{equ: main 2} subject to the exterior Dirichlet data $f$.
			\end{lemma}
			
			\begin{proof}
			We give a brief sketch of this standard lemma. By considering $u=v+f$ in $\R^n$, we can study the well-posedness of an alternative equation 
				\begin{equation}
					\begin{cases}
						P_q v = \varphi  &\text{ in }\Omega, \\
						v=0 &\text{ in }\Omega_e,
					\end{cases}
				\end{equation}
				where $\varphi = \left. - (P_q f) \right|_{\Omega}$. Recalling that $\{b_k\}_{k=1}^\infty\subset (0,\infty)$, consider the sesquilinear form
				\begin{equation}
					B_0 (v,w):=\sum_{k=1}^N b_k \big(  (-\Delta)^{s_k/2} v , (-\Delta)^{s_k/2} \big)_{L^2(\R^n)}.
				\end{equation}
				for any $v,w\in \wt H^{s_N}(\Omega)$.
				It is straightforward to see that $B_0(\cdot, \cdot)$ satisfies both boundedness and coercive. In other words, we have 
				\begin{equation}
					\left| B_0(v,w)\right| \leq \sum_{k=1}^N \big\| (-\Delta)^{s_k/2} v \big\|_{L^2(\R^n)}\big\| (-\Delta)^{s_k/2} w \big\|_{L^2(\R^n)},
				\end{equation}
				and 
				\begin{equation}
					B_0(v,v)\geq \sum_{k=1}^N b_k \big\| (-\Delta)^{s_k/2} v \big\|_{L^2(\R^n)}^2 \geq b_N \big\| (-\Delta)^{s_N/2} v \big\|_{L^2(\R^n)}^2 ,
				\end{equation}
				for any $v,w\in \wt H^{s_N}(\Omega)$, where we use $\{ b_k\}_{k=1}^N\subset (0,\infty)$ in the last inequality. Hence, the rest of the proof follows the standard method of the proof of the Lax-Milgram theorem (for example, see \cite{GSU20,GLX}), and is therefore omitted.  In short, for the sesquilinear form 
				\begin{equation}\label{B_q bilinear}
					B_q(v,w):=B_0(v,w)+\LC qv , w \RC_{L^2(\Omega)},
				\end{equation}
				one can find a unique $v\in \wt H^s(\Omega)$ satisfying 
				$$
				B_q(v,w)=\langle \varphi , w \rangle_{H^{-s_N}(\Omega),\wt H^{s_N}(\Omega)},
				$$ 
				for any $w\in \wt H^{s_N}(\Omega)$, provided the condition \eqref{eigenvalue condition} holds. This concludes the proof.
			\end{proof}

			With the above-mentioned well-posedness result, the DN map \eqref{DN map} is well-defined. Specifically, there is the relation
			\begin{equation}\label{DN and bilinear}
				\left\langle \Lambda_q f ,g \right\rangle :=\left\langle \Lambda_q f ,g \right\rangle_{H^{-s_N}(\Omega_e),\widetilde{H}^{s_N}(\Omega_e)} = B_q (u_f,w_g),
			\end{equation}
			where $u_f\in H^{s_N}(\R^n)$ is the solution to \eqref{equ: main 2}, $w_g\in H^{s_N}(\R^n)$ is any function with $\left. w_g \right|_{\Omega_e}=g$ and $B_q(\cdot, \cdot)$ is given by \eqref{B_q bilinear}.
			Furthermore, one can derive the following integral identity.
			
			\begin{lemma}[Integral identity]
				Let $\Omega\subset \R^n$ be a bounded domain with Lipschitz boundary for $n\geq 2$. Let $\{b_k\}_{k=1}^N\subset (0,\infty)$ and $\{s_k\}_{k=1}^N\subset (0,\infty)$ with $0<s_1<s_2<\ldots< s_N$ satisfy \ref{exponent condition}. Let $q, q_{1},q_{2}\in L^{\infty}(\Omega)$
				satisfy \eqref{eigenvalue condition}. For any $f_1,f_2\in C^\infty_0(\Omega_{e})$, we have the symmetry property
				\begin{equation}\label{eq:adjoint operator}
					\left\langle \Lambda_q f_1 , \overline{f_2}\right\rangle = 	\left\langle f_1, \Lambda_{\overline q} f_2\right\rangle ,
				\end{equation}
				and the integral identity
				\begin{equation}\label{eq:integral identity}
					\left\langle (\Lambda_{q_{1}}-\Lambda_{q_{2}})f_{1},f_{2}\right\rangle=\LC \LC q_{1}-q_{2}\RC u_{f_1},u_{f_2}\RC_{L^2(\Omega)}
				\end{equation}
				where for $j=1,2,$ $u_{f_j}\in H^{s}(\mathbb{R}^{n})$ is the unique solution to \eqref{equ: main 2} with $q=q_j$ and $f=f_j$.
			\end{lemma}
			
			\begin{proof}
				The symmetry \eqref{eq:adjoint operator} of the DN map comes from the symmetry of the sesquilinear form $B_q(\cdot, \cdot)$ (see e.g. \eqref{DN and bilinear}). On the other hand, by \eqref{eq:adjoint operator}, we have 
				\begin{equation*}
					\begin{split}
						\left\langle (\Lambda_{q_{1}}-\Lambda_{q_{2}})f_{1},\overline{f_{2}}\right\rangle & =\left\langle \Lambda_{q_{1}}f_{1},\overline{f_{2}}\right\rangle-\left\langle f_{1},\Lambda_{\overline{q_{2}}}\overline{f_{2}}\right\rangle\\\
						&=B_{q_{1}}(u_{f_1},u_{f_2})-B_{q_{2}}(u_{f_1},u_{f_2})\\
						& =\LC \LC q_{1}-q_{2}\RC u_{f_1},u_{f_2}\RC_{L^2(\Omega)}.
					\end{split}
				\end{equation*}
				This concludes the proof.
			\end{proof}

			\section{Entanglement principle}\label{sec: entanglement}

			We first show that the proof of Theorem~\ref{thm: ent} follows from an analogous statement for smooth functions whose derivatives of all orders enjoy super-exponential decay at infinity. 
			
			\begin{theorem}\label{thm_ent_smooth}
		      Let $\{\alpha_k\}_{k=1}^N\subset (0,1)$ with $\alpha_1<\ldots<\alpha_N$ satisfy \begin{equation}
					\label{exp_condition_alpha}
					\left(|\alpha_j-\alpha_k|\neq \frac{1}{2} \quad \text{for $j,k=1,\ldots,N$} \right), \quad \text{if the dimension $n$ is odd.}
				\end{equation}
				Let $\mathcal{O}\subset \R^n$, $n\geq 2$, be a nonempty open set and assume that $\{v_k\}_{k=1}^N\subset C^{\infty}(\R^n)$ and that there exists constants $\rho>0$ and $\gamma>1$ such that given any multi-index $\beta=(\beta_1,\ldots,\beta_n) \in \LC \N \cup \{0\} \RC^n$ there holds 
				\begin{equation}\label{exp_decay}
					\left|D^{\beta} v_k(x)\right| \leq C_\beta\, e^{-\rho|x|^\gamma} \quad \forall\, x\in \R^n \qquad k=1,\ldots,N,
				\end{equation}
				for some $C_\beta>0$ where $D^\beta = \frac{\p^{|\beta|}}{\p x_1^{\beta_1}\ldots\, \p x_n^{\beta_n}}.$
				If,
				\begin{align}\label{condition_ent}
					v_1|_{\mathcal O}=\ldots=v_N|_{\mathcal O}=0 \quad \text{and} \quad  \sum_{k=1}^N ((-\Delta)^{\alpha_k}v_k)\big|_{\mathcal O}=0,
				\end{align} 
			then $v_k\equiv 0$ in $\R^n$ for each $k=1,\ldots,N$.
			\end{theorem}
			At first glance, the above theorem may appear slightly weaker than our entanglement principle of Theorem \ref{thm: ent}, since it imposes more regularity and decay on the functions and the exponents are restricted to $(0,1)$. As we will see in a moment, a mollifier argument allows us to show that Theorem~\ref{thm: ent} can be proven from this weaker version. 
			\begin{proof}[Proof of Theorem~\ref{thm: ent} via Theorem~\ref{thm_ent_smooth}]
				We will assume that the hypothesis of Theorem~\ref{thm: ent} is satisfied. Assume without loss of generality that $\{u_k\}_{k=1}^N\subset H^{-r}(\R^n)$ for some $r\in \R$. Let $\phi \in C^{\infty}_0(\R^n)$ be a nonnegative function with compact support inside the open unit ball centered at the origin such that $\|\phi\|_{L^1(\R^n)}=1$. We fix a nonempty open set $\widetilde{\mathcal O} \Subset O$ so that
				\begin{equation}
					\label{tilde_O}
					\textrm{dist}(x, \R^n\setminus \mathcal O)>\epsilon_0 \qquad \forall\, x\in \overline{\widetilde{\mathcal O}}
				\end{equation}
				for some $\epsilon_0\in (0,1).$ Define, for each $\epsilon \in (0,\epsilon_0)$ the function 
				$$\psi_\epsilon(x) := \epsilon^{-n}\phi(\epsilon^{-1}x).$$
				Next, we define for each $x\in \R^n$, and each $\epsilon \in (0,\epsilon_0)$, the function $\wt v_{k,\epsilon}\in C^{\infty}(\R^n)$ by
				$$ \wt v_{k,\epsilon}(x)=  b_k\,\LC u_k\ast \psi_\epsilon\RC(x):=b_k \langle u_k(\cdot) ,\psi_\epsilon(x-\cdot)\rangle\quad k=1,\ldots,N,$$
				where $\langle \cdot,\cdot\rangle$ denotes the sesquilinear pairing between $H^{-r}(\R^n)$ and $H^r(\R^n)$ as explained in Section~\ref{sec: preliminary: fcn}. As $u_k$ with $k=1,\ldots,N$ all vanish on $\mathcal O$, we obtain in view of \eqref{tilde_O} that 
				\begin{equation}
					\label{v_k_zero}
					\wt v_{k,\epsilon}(x)=0 \quad \forall\, x\in \widetilde{\mathcal O} \quad \epsilon \in (0,\epsilon_0) \quad k=1,\ldots,N.
				\end{equation}
				Furthermore, given any multi-index $\beta \in \LC \N \cup \{0\}\RC ^n$ and in view of the fact that the distributions $\{u_k\}_{k=1}^N$ all have super-exponential decay in the sense of Definition~\ref{def_exp}, we obtain for each $x\in \R^n$ with $|x|>2$ and each $k=1,\ldots,N$,
				$$
				\left|D^\beta \wt v_{k,\epsilon}(x)\right| = \left| b_k \langle u_k,D^\beta \psi_\epsilon(x-\cdot)\rangle\right| \leq \left|b_k\right|\,C \,e^{-\rho \,(|x|-1)^\gamma}  \left\|\psi_\epsilon\right\|_{H^{r+|\beta|}(\R^n)},
				$$
				where we used the fact that $\psi_{\epsilon}(x-\cdot)$ is supported outside the closed ball $B_{|x|-1}(0)$ together with Definition~\ref{def_exp} with the choice $R=|x|-1$. Therefore, by modifying the constant $C>0$ above we deduce that there exists $C_\beta>0$ (depending on $\beta$ and $\epsilon$) such that
				\begin{equation}\label{v_beta_decay}
					\left|D^\beta \wt v_{k,\epsilon}(x)\right| \leq C_\beta \,e^{-{\rho 2^{-\gamma}\, |x|^\gamma}}, \quad \text{for all $x\in \R^n$ and all $k=1,\ldots,N.$}
				\end{equation}

				Next, let us write 
				$
				s_k = \lfloor s_k \rfloor +\alpha_k,
				$
				where $\lfloor s_k\rfloor$ is the greatest integer not exceeding $s_k$ and $\alpha_k \in (0,1)$ is its fractional part. The reason that the fractional parts $\alpha_k$ are never zero here is due to \ref{exponent condition}. Define
				$$
				v_{k,\epsilon}(x) = b_k \,(-\Delta)^{\lfloor s_k\rfloor}\wt{v}_{k,\epsilon}
				\quad k=1,\ldots,N \quad \epsilon \in (0,\epsilon_0).       $$
				It is now straightforward to see that the hypothesis of Theorem~\ref{thm_ent_smooth} is satisfied with $\{v_k\}_{k=1}^N$ in its statement replaced with the functions $\{v_{k,\epsilon}\}_{k=1}^N$ and with $\mathcal O$ in its statement replaced with $\widetilde{\mathcal O}$. Indeed, thanks to \eqref{v_beta_decay}, we see that these functions enjoy the super-exponential decays stated in \eqref{exp_decay} and also that they satisfy the condition \eqref{condition_ent}. Moreover, by \ref{exponent condition}, the fractional parts of $s_k$ all belong to $(0,1)$ and additionally satisfy \eqref{exp_condition_alpha}. Thus, applying Theorem~\ref{thm_ent_smooth} to these functions, we conclude that there holds
				$$
				(-\Delta)^{\lfloor s_k\rfloor}\wt{v}_{k,\epsilon}=0 \quad \text{in $\R^n$ for all $k=1,\ldots,N$.}
				$$
					The latter equation implies that $\tilde{v}_{k,\epsilon}$ is identical to zero. Indeed, this is trivial to see if $\lfloor s_k\rfloor=0$ and in the other case that $\lfloor s_k\rfloor\in \N$, it follows from applying the unique continuation principle for the Laplace operator on $\R^n$. Therefore, 
				$$ 
				\langle u_k(\cdot), \psi_\epsilon(x-\cdot)\rangle=0 \quad \text{in $\R^n$ and all $k=1,\ldots,N$.} 
				$$
				Finally, we obtain the desired claim by letting $\epsilon$ approach zero and noting that $b_k\neq 0$ for $k=1,\ldots,N$.
			\end{proof}
			
			The rest of this section is concerned with proving Theorem~\ref{thm_ent_smooth}. Thus, we will assume throughout the remainder of the section that $\{\alpha_k\}_{k=1}^N\subset (0,1)$ and that $\{v_k\}_{k=1}^N\subset C^{\infty}(\R^n)$ are as stated in the hypothesis of Theorem~\ref{thm_ent_smooth}. In the rest of this section, we make the standing assumption that $\omega \Subset \mathcal O$ is a fixed nonempty bounded open set and that there holds 
			\begin{equation}\label{kappa}
				\textrm{dist}(\overline{\omega},\R^n\setminus \mathcal O)\geq 2\kappa >0,
			\end{equation}
			for some constant $\kappa\in (0,1)$. For our purposes, it suffices to think of $\omega$ as a sufficiently small neighbourhood of some fixed point inside $\mathcal O$.
			
			\subsection{Analytic interpolation in the right half-plane}

			We remark that throughout the remainder of Section~\ref{sec: entanglement}, the notation $\log(z)$, $z\in \C$ stands for the principal branch of the logarithm function. Also, given any $a>0$, the notation $a^z$ stands for $e^{a\log z}$. We begin this section with a few lemmas.
			
			\begin{lemma}\label{lem_analytic}
				Given each fixed $x\in \omega$, the function 
				$$ F: \{ z\in \C\,:\, \mathrm{Re}(z) \geq 0\} \to \C$$
				defined via
				\begin{equation}\label{F(z)}
					F(z):=  \sum_{k=1}^N \frac{\Gamma(z+1+\alpha_k)}{\Gamma(-\alpha_k)\Gamma(1+\alpha_k)}\int_0^\infty (e^{t\Delta} v_k)(x)\, t^{-(z+1+\alpha_k)}\, dt
				\end{equation}
				is holomorphic.
			\end{lemma}
			\begin{proof}
				As each of the functions $z\mapsto \Gamma(z+1+\alpha_k)$, $k=1,\ldots,N$ are holomorphic in the right half-plane $\left\{ z\in \C: \, \mathrm{Re}(z)\geq 0 \right\}$, it suffices to show that for each $k=1,\ldots,N$ the function 
				\begin{equation}\label{g_def}
					G_k(z):= \int_0^{\infty} (e^{t\Delta}v_k)(x)\,t^{-(z+1+\alpha_k)}\,dt,\end{equation}
				is also holomorphic in the right half-plane. It is straightforward to see that for each $z$ in the right half-plane the above integrands are absolutely integrable. Our task is to show that they depend analytically on $z$. We write
				\begin{equation}\label{g_rand} 
					\begin{split}
						G_k(z) := G_{1,k}(z)+G_{2,k}(z):= \int_0^{1} (e^{t\Delta}v_k)(x)\,t^{-(z+1+\alpha_k)}\,dt + \int_1^{\infty} (e^{t\Delta}v_k)(x)\,t^{-(z+1+\alpha_k)}\,dt,
					\end{split}
				\end{equation}
				and proceed to show that for each fixed $x\in \omega$, each of $G_{j,k}$, $j=1,2$, $k=1,\ldots, N$ depend analytically on $z$ with $\textrm{Re}(z)\geq 0$. We begin by noting that given any $h\in \C$ and any $t>0$ there holds
				\begin{equation}\label{complex_plane_bound}
					\left|t^{-h}-1+h\log(t)\right| \leq \frac{1}{2} (\log t)^2\, |h|^2\, e^{|h\log t|} \quad \forall\, h\in \C,
				\end{equation}
				where we used Taylor series approximation for $e^{-h\,\log t}$ around $h=0$ with two terms. Recalling that $\omega \Subset \mathcal O$ and the bound \eqref{kappa}, we also record the following straightforward point-wise bound for heat semigroups on $\R^n$,
				\begin{equation}
					\label{heat_bound_rand}
					\left|e^{t\Delta}v_k(x)\right| \leq \frac{\left\|v_k\right\|_{L^1(\R^n)} }{(4\pi t)^{\frac{n}{2}}}\, e^{-\frac{\kappa^2}{t}}\, \qquad t>0\quad x\in \omega.
				\end{equation}

				For the function $G_{1,k}$ given by \eqref{g_rand}, we note by using \eqref{complex_plane_bound} that given each $z\in \C$ and each $h\in \C$ with $|h|< 1$, we have
				\begin{equation}
					\begin{split}
						& \bigg|G_{1,k}(z+h)- G_{1,k}(z) + h \underbrace{\int_0^1 (e^{t\Delta}v_k)(x)\,t^{-(z+1+\alpha_k)}\,(\log t)\,dt}_{\textrm{Absolutely convergent for all $z\in \C$}} \bigg| \\
						&\qquad \leq \frac{\|v_k\|_{L^1(\R^n)}}{2(4\pi)^{\frac{n}{2}}}|h|^2 \underbrace{\int_0^1 (\log t)^2\,t^{-\frac{n}{2}-2-|z|-\alpha_k}\, e^{-\frac{\kappa^2}{t}}\,dt}_{\textrm{Absolutely convergent for all $z\in \C$}}
					\end{split}    
				\end{equation}
				showing in fact that the function $G_{1,k}(z)$ is holomorphic everywhere on $\C$. On the other hand, for the function $G_{2,k}(z)$ we note that given each $z$ with $\textrm{Re}(z)\geq 0$ and each $|h|<1$ there holds
				\begin{equation}
					\begin{split}
						&   \bigg|G_{2,k}(z+h)- G_{2,k}(z) + h \int_1^\infty (e^{t\Delta}v_k)(x)\,t^{-(z+1+\alpha_k)}\,(\log t)\,dt \bigg| \\
						&\qquad \leq \frac{\|v_k\|_{L^1(\R^n)}}{2(4\pi)^{\frac{n}{2}}}|h|^2 \int_1^\infty (\log t)^2 t^{-\frac{n}{2}-\alpha_k}\,dt
					\end{split}
				\end{equation}
				showing indeed that the function $G_{2,k}$ is holomorphic on $\textrm{Re}(z)\geq 0$.
			\end{proof}
			
			\begin{lemma}\label{lem_meromorphic}
				Given each fixed $x\in \Omega$, the function $F$ defined by \eqref{F(z)} admits a meromorphic extension to $\C$ (with isolated poles of order at most two) given by 
				\begin{equation}\label{F_def}
					F(z)= \sum_{k=1}^N \frac{4^{\alpha_k+z}}{\pi^{\frac{n}{2}}} \, \frac{\Gamma\LC z+1+\alpha_k\RC  \Gamma \LC z +\frac{n}{2}+\alpha_k\RC}{\Gamma(-\alpha_k)\,\Gamma(1+\alpha_k)} \int_{\R^n\setminus \mathcal{O}} \frac{v_k(y)}{\abs{x-y}^{n+2\alpha_k +2z}}\, dy.
				\end{equation}
			\end{lemma}
			
			\begin{proof}
				By direct computation, we obtain for each $k=1,\ldots,N$ and each $z\in \C$ with $\mathrm{Re}(z)\geq 0$ that
				\begin{equation}
					\begin{split}
						&\quad \,  \int_0^\infty  \LC e^{t\Delta}v_k\RC(x) e^{-(z+\alpha_k+1)}\, dt \\
						&= \underbrace{\frac{1}{(4\pi)^{n/2}}\int_0^\infty \LC \int_{\R^n\setminus \mathcal{O}}\frac{1}{t^{n/2}}e^{-\frac{\abs{x-y}^2}{4t}} v_k(y)\, dy\RC t^{-(\alpha_k+z+1)}\, dt}_{\text{By Fubini's theorem, heat kernel formula \eqref{heat kernel representation} and }v_k=0 \text{ in }\mathcal{O}} \\
						&= \underbrace{\int_{\R^n\setminus \mathcal O} \left(\int_0^\infty t^{-(\frac{n}{2}+\alpha_k+z+1)} e^{-\frac{\abs{x-y}^2}{4t}}\, dt\right)v_k(y)\,dy}_{\text{Absolutely convergent for }\mathrm{Re}(z)\geq 0}  \\
						&= \frac{4^{\alpha_k+z}}{\pi^{\frac{n}{2}}}  \underbrace{\bigg( \int_0^\infty t^{\frac{n}{2}+\alpha_k+z-1}\,e^{-t}\, dt\bigg) }_{\text{By change of variable }} \underbrace{\int_{\R^n\setminus \mathcal{O}} \frac{v_k(y)}{|x-y|^{n+2\alpha_k+2z}}\, dy}_{\textrm{Absolutely convergent for all $z\in \C$ by} \, \eqref{exp_decay}}\\  
						&= \frac{4^{\alpha_k+z}}{\pi^{\frac{n}{2}}} \Gamma\LC z+\frac{n}{2}+\alpha_k\RC \int_{\R^n\setminus \mathcal{O}} \frac{v_k(y)}{|x-y|^{n+2\alpha_k+2z}}\, dy,
					\end{split}
				\end{equation}
				where we use $x\in \omega$ and $y\in \R^n\setminus \mathcal{O}$ so that \eqref{kappa} holds. 
				
				Next, we aim to show that the previous expression derived for $\mathrm{Re}(z)\geq 0$ can be viewed on the entire complex plane. Indeed, analogously to the proof of Lemma~\ref{lem_analytic} and due to the fast  decay of $v_k$ given by \eqref{exp_decay} we have that the mapping 
				$$ z \mapsto \int_{\R^n\setminus \mathcal{O}} \frac{v_k(y)}{|x-y|^{n+2\alpha_k+2z}}\, dy, \quad k=1,\ldots,N, \quad x\in \omega$$
				is holomorphic for all $z\in \C$ (let us point out that only Schwartz decay for $v_k$ is needed here). The Gamma function is also meromorphic on the entire complex plane. We have thus completed the proof of the lemma. 
			\end{proof}
			
			The following lemma also appears in \cite{FKU24} and is a consequence of the upper bound for the Gamma function, see \cite[formula (2.1.19) on page 34]{Paris_Kaminski_book}; see also \cite[page 300]{Olverbook},
			\begin{equation}
				\label{eq_bound_Gamma_f}
				|\Gamma(z)| \le \sqrt{2\pi} |z|^{a-\frac{1}{2}} e^{-\frac{\pi}{2}|b|} e^{\frac{1}{6} |z|^{-1}},
			\end{equation}
			valid for all $ z = a + ib \in \C$, where $a\ge 0$.
			
			\begin{lemma}[Lemma 3.4 in \cite{FKU24}]\label{lem_H_bounds}
				Let $C:=\sqrt{2\pi}e^{\frac{1}{6}}$. Given $s \in (0,\infty)$, the function 
				$$ H(z):=\Gamma(z+1+s),$$
				is meromorphic on $\mathbb{C}$, with its only singularities being simple poles at $z=-k-1-s$ for $k=0,1,2,\dots$, and additionally satisfies 
				\begin{enumerate}[(1)]
					\item\label{item 1 lem 3.4} {$|H(\textrm{\em i}b)|\leq C (2|b|)^{s+\frac{1}{2}} e^{-\frac{\pi}{2}|b|}$ for all $b\in \R$ such that $|b|\ge 1+s$,}
					\item\label{item 2 lem 3.4} {$|H(a)| \leq C e^{a \log a}e^{a\log 2+(s+\frac{1}{2})\log(2a)}$ for all $a\ge 1+s$,} 
					\item\label{item 3 lem 3.4} {$|H(z)|\leq C e^{2|z|\log(2|z|)}$ for all $z\in \{a+\textrm{\em i}b\in \C :\, a\ge  0,\  b\in \R\}$ such that $|z|\ge 1+s$}.
				\end{enumerate} 
			\end{lemma}
			
			\begin{lemma}[cf. Lemma 3.5 \cite{FKU24}]
				\label{lem_analytic_bounds}
				Given each $k=1,\ldots,N$ and each fixed $x\in \omega$,
				there exist constants $c_1, c_2 > 0$, depending on $\kappa\in (0,1)$ in \eqref{kappa}, $\|v_k\|_{L^1(\R^n)}$ and on $\alpha_k$ such that the holomorphic function $G_k: \{\mathrm{Re}(z)\geq 0\}\to \C$ defined by \eqref{g_def} satisfies the bounds
				\begin{enumerate}[(i)]
					\item\label{item i G_k} $\left|G_k(\mathrm{i}b)\right| \leq c_1$ for all $b \in \mathbb{R}$,
					\item\label{item ii G_k} $\left|G_k(a)\right| \leq c_1\,e^{c_2 a}\, e^{a \log a}$ for all $a \ge 1$,
					\item\label{item iii G_k} $\left|G_k(z)\right| \leq c_1\,e^{c_2|z|} e^{2|z| |\log z|}$ for all $z \in \{a + \mathrm{i}b \in \mathbb{C}: \, a \ge 0,\  b \in \mathbb{R}\}$ with $|z| \ge 1$.
				\end{enumerate}
			\end{lemma}
			
			\begin{proof}
				We first note that $t^{-z} = e^{-z \log t}$ for $t > 0$ where the logarithm is defined using its principal branch. By definition of $G_k$ together with \eqref{kappa} and \eqref{heat_bound_rand}, we write for each $z=a+\mathrm{i}b$ with $a\geq 0$ that
                \begin{equation}
					\begin{split}
						|G_k(z)|&\leq  \int_0^1  |(e^{t\Delta}v_k)(x)|\,t^{-(a+1+\alpha_k)}\,dt+ \int_1^{\infty}|(e^{t\Delta}v_k)(x)|\,t^{-(a+1+\alpha_k)}\,dt 
						\\
						&\leq  \int_0^1 \frac{\|v_k\|_{L^1(\R^n)} }{(4\pi t)^{\frac{n}{2}}}\, e^{-\frac{\kappa^2}{t}} \,t^{-(a+1+\alpha_k)}\,dt+ \int_1^{\infty}\frac{\|v_k\|_{L^1(\R^n)} }{(4\pi t)^{\frac{n}{2}}}\, e^{-\frac{\kappa^2}{t}}t^{-(a+1+\alpha_k)}\,dt \\
						&\leq \underbrace{\frac{\|v_k\|_{L^1(\R^n)} }{(4\pi)^{\frac{n}{2}}} \int_1^\infty e^{-\kappa^2 t} \,t^{a+\alpha_k+\frac{n}{2}-1}\,dt}_{\text{Change of variables}}+ \frac{\|v_k\|_{L^1(\R^n)} }{(4\pi)^{\frac{n}{2}}} \int_1^\infty t^{-(a+\alpha_k+\frac{n}{2}+1)}\,dt\\
						&=\frac{\|v_k\|_{L^1(\R^n)} }{(4\pi)^{\frac{n}{2}}\kappa^{2a+2\alpha_k+n}} \int_{\kappa^2}^\infty e^{-t} \,t^{a+\alpha_k+\frac{n}{2}-1}\,dt+ \frac{\|v_k\|_{L^1(\R^n)} }{(4\pi)^{\frac{n}{2}}} \int_1^\infty t^{-(a+\alpha_k+\frac{n}{2}+1)}\,dt\\
						&\leq \frac{\|v_k\|_{L^1(\R^n)}}{(4\pi)^{\frac{n}{2}}}\left(\kappa^{-2a-2\alpha_k-n}\Gamma\LC a+\alpha_k+\frac{n}{2}\RC+ \frac{1}{\frac{n}{2}+1+a}\right)\\
						&\leq 2\frac{\|v_k\|_{L^1(\R^n)}}{(4\pi)^{\frac{n}{2}}}\,\kappa^{-2a-2\alpha_k-n}\,\Gamma\LC a+\alpha_k+\frac{n}{2}\RC.
					\end{split}
				\end{equation}
				All the bounds \ref{item i G_k}--\ref{item iii G_k} follow immediately from combining the previous bound with Lemma~\ref{lem_H_bounds} \ref{item 1 lem 3.4}--\ref{item 3 lem 3.4}. This completes the proof.
			\end{proof}

			The following proposition is a key step in the proof of Theorem~\ref{thm_ent_smooth}.

			\begin{proposition}\label{prop_F_zero}
				Let the hypothesis of Theorem~\ref{thm_ent_smooth} be satisfied. Given each $x\in \omega$, the meromorphic function $F$ defined by \eqref{F_def} vanishes everywhere in the complex plane, away from its isolated poles.  
			\end{proposition}
			
			In order to prove the above proposition, we need to make use of a sharp interpolation theorem for holomorphic functions that vanish on positive integers subject to certain growth rates at infinity, due to Pila \cite{pila_05}. For the convenience of the reader, we include Pila's theorem here.
			\begin{theorem}
				\label{thm_Pila}
				Let $\alpha,\beta\in \R$ with $\alpha+\beta<1$  and let $\epsilon>0$. Write $z=a+ib$ and suppose that $h(z)$ is holomorphic in the region $a\ge 0$, satisfying:
				\begin{enumerate}[(i)]
					\item {$\limsup\limits_{|b|\to \infty}\frac{\log|h(ib)|}{\pi |b|}\le \beta$,}
					\item {$\limsup\limits_{a\to \infty}\frac{\log|h(a)|}{2a\log a}\le \alpha$,} 
					\item {$\log|h(z)|=\mathcal{O}(|z|^{2-\epsilon})$, throughout $a\ge 0$, as $|z|\to\infty$}. 
				\end{enumerate}
				Suppose that $h(m)=0$ for all $m\in \N$. Then $h\equiv 0$ on the set $\{ z\in \C: \, \mathrm{Re}(z)\geq 0\}$. 
			\end{theorem}
			
			\begin{proof}[Proof of Proposition~\ref{prop_F_zero}]
				
				The proof is analogous to the proof of \cite[Lemma 3.7]{FKU24} but with some modifications as we are working on $\R^n$. Throughout this proof, it is important to recall that the elements of $\{v_k\}_{k=1}^{N}$ are smooth and in fact belong to the Schwartz class $\mathcal S(\R^n)$.  We begin by noting that the condition \eqref{condition_ent} implies in particular that 
				\begin{equation}
					\begin{split}
						\Delta^{m}v_1 =\ldots = \Delta^{m}v_k=0 \text{ in }\mathcal{O},\quad m\in \N,
					\end{split}
				\end{equation}
				and subsequently that
				\begin{equation}
					\sum_{k=1}^N  (-\Delta)^{\alpha_k}\Delta^{m} v_k=0 \text{ in }\mathcal{O},
				\end{equation}
				for any $m\in \N$, where we use $(-\Delta)^{\alpha+\beta}=(-\Delta)^\alpha (-\Delta)^{\beta}$ on $H^{2\alpha+2\beta}(\R^n)$, for all $\alpha, \beta \in \R$. Recalling that $\omega \Subset \mathcal{O}$ is a nonempty bounded open set and that $\kappa\in (0,1)$ is as in \eqref{kappa}, the above identity implies that 
				\begin{align}\label{pf of UCP 1}
					\sum_{k=1}^N \frac{1}{\Gamma(-\alpha_k)} \int_0^\infty  \frac{e^{t\Delta}\Delta^{m} v_k(x)}{t^{1+\alpha_k}}\, dt=0 ,\text{ for }x\in \omega,
				\end{align}
				for all $m\in \N$. We claim that the previous identity \eqref{pf of UCP 1} can be reduced to 
				\begin{equation}\label{eq_important_m}
					\sum_{k=1}^N \frac{\Gamma(1+m+\alpha_k)}{\Gamma(-\alpha_k)\Gamma(1+\alpha_k)} \int_0^{\infty} (e^{t\Delta}v_k)(x)\,t^{-(1+m+\alpha_k)}\,dt =0 \quad \forall\, x\in \omega\quad \forall m\in \N. 
				\end{equation}
				In order to prove \eqref{eq_important_m} from \eqref{pf of UCP 1} we will employ an integration by parts trick that has been used recently in several related works, see for example, the proofs of \cite[Lemma 3.7]{FKU24}, \cite[Proposition 3.1]{GU2021calder} and \cite[Theorem 1.1]{FGKU_2025fractional}). Using the fact that $t\mapsto e^{t\Delta} (\Delta v_k)\in C^\infty([0,\infty); C^\infty(\R^n))$, and  that $e^{t\Delta}\Delta^{m}= \Delta^{m} e^{t\Delta}$ for all $t\ge 0$ on $\mathcal{D}(\Delta^m)=H^{2m}(\R^n)$, we obtain for any $m\in \N$,
				\begin{equation}
					\label{eq_100_4}
					\big(e^{t\Delta}\Delta^{m} v_k\big)(x)=\p_t^{m} \big(e^{t\Delta}v_k\big)(x), 
				\end{equation}
				for  $x\in \omega$ and $k=1,\dots, N$. Combining \eqref{pf of UCP 1} and  \eqref{eq_100_4}, we get 
				\begin{equation}
					\label{eq_100_5}
					\sum_{k=1}^N \frac{1}{\Gamma(-\alpha_k)}\int_0^\infty  \p_t^m(e^{t\Delta}v_k)(x) \frac{dt}{t^{1+\alpha_k}}=0,
				\end{equation}
				for $x\in \omega$ and all $m\in \N$. We shall next repeatedly integrate by parts in \eqref{eq_100_5} $m$ times and show that no contributions arise from the endpoints of the integral. Indeed, for $t > 0$ and $x \in \omega$, using \eqref{eq_100_4} and \eqref{exp_decay}, we obtain (analogously as in \eqref{heat_bound_rand}) that
				\begin{equation}
					\label{eq_100_5_1}
					\left|\p_t^l (e^{t\Delta}v_k)(x)\right|=\left|(e^{t\Delta}\Delta^{l}v_k)(x)\right|\le  \frac{\left\|\Delta^l v_k\right\|_{L^1(\R^n)} }{(4\pi t)^{\frac{n}{2}}}\, e^{-\frac{\kappa^2}{t}}\, \qquad t>0\quad x\in \omega,
				\end{equation}
				where $l = 0, 1, \dots, m-1$. The bound \eqref{eq_100_5_1} shows that no contribution from $t =0$ or $t=\infty$ arises when integrating by parts in \eqref{eq_100_5}. Thus, by integrating by parts $m$ times in \eqref{eq_100_5}, we obtain that 
				\begin{equation}
					\label{eq_100_8}
					\sum_{k=1}^N \frac{1}{\Gamma(-\alpha_k)}c_k\int_0^\infty  (e^{t\Delta} v_k)(x) \frac{dt}{t^{m+1+\alpha_k}}=0,
				\end{equation}
				for $x\in \omega$ and all $m\in \N$.  Here 
				\begin{equation}
					\label{eq_100_9}
					c_k=(1+\alpha_k)(2+\alpha_k)\dots (m+\alpha_k)=\frac{\Gamma(m+1+\alpha_k)}{\Gamma(1+\alpha_k)}.
				\end{equation}
				This completes the proof of \eqref{eq_important_m}. Recalling the definition \eqref{F(z)}, we observe next that \eqref{eq_important_m} may be rewritten in the form
				\begin{equation}
					\label{F(m)}
					F(m)=0, \quad \text{for all } m\in \N.
				\end{equation}
				Recall from Lemma~\ref{lem_analytic} that $F(z)$ is holomorphic in $\left\{ z: \, \mathrm{Re}(z)\geq 0\right\}$ and also that it admits a meromorphic extension to all of $\C$ via Lemma~\ref{lem_meromorphic}. Furthermore, recalling that
				$$ F(z)= \sum_{k=1}^N \frac{\Gamma(z+1+\alpha_k)}{\Gamma(-\alpha_k)\Gamma(1+\alpha_k)}\,G_k(z),$$
				it follows from Lemma~\ref{lem_H_bounds} and Lemma~\ref{lem_analytic_bounds} that the function $F(z)$  satisfies the bounds on the set $\left\{ z\in \C: \, \mathrm{Re}(z)\geq 0\right\}$ stated in Theorem~\ref{thm_Pila} (with $h$ replaced by $F$ and with the choices $\alpha=1$ and $\beta=-\frac{1}{2}$). Thus, we conclude from Pila's theorem that $F(z)$ on the right half-plane and subsequently by analytic continuation that its meromorphic extension given by Lemma~\ref{lem_meromorphic} vanishes identically on $\C$ away from its isolated poles. These poles happen precisely at those values of $z$ for which either $z+1+\alpha_k$ or $z+\frac{n}{2}+\alpha_k$ is a nonpositive integer for some $k=1,\ldots,N$.
			\end{proof}     
			
			\subsection{Analysis of singularities of the meromorphic extension}
			We aim to use Proposition~\ref{prop_F_zero} and perform an analysis of the poles of the function $F(z)$. Recall from Lemma~\ref{lem_meromorphic} that given a fixed $x\in \omega$, the meromorphic function $F:\C \to \C$ is given by the expression 
			$$ F(z)= \sum_{k=1}^N \frac{4^{\alpha_k+z}}{\pi^{\frac{n}{2}}} \, \frac{\Gamma\LC z+1+\alpha_k\RC  \Gamma \LC z +\frac{n}{2} +\alpha_k \RC}{\Gamma(-\alpha_k)\,\Gamma(1+\alpha_k)} \int_{\R^n\setminus \mathcal{O}} \frac{v_k(y)}{\abs{x-y}^{n+2\alpha_k +2z}}\, dy.$$
			Based on the proof of Lemma~\ref{lem_meromorphic} we also know that a point $z\in \C$ is a pole of the function $F$ above if and only if it is of the form $z=-\alpha_k-m-\frac{n}{2}$ or $z=-\alpha_k-m-1$ for some $k=1,\ldots,N$ and some $m=0,1,2,\ldots$. Our analysis will need to be modified depending on the parity of the dimension $n \geq 2$. \\
			
			\noindent{\bf \em II.I. case of even dimension.} We will assume first that $n=2p$ for some $p\in \N$ so that
			$$ F(z)= \sum_{k=1}^N \frac{4^{\alpha_k+z}}{\pi^{p}} \, \frac{\Gamma\LC z+1+\alpha_k\RC  \Gamma \LC z +p +\alpha_k\RC}{\Gamma(-\alpha_k)\,\Gamma(1+\alpha_k)} \int_{\R^n\setminus \mathcal{O}} \frac{v_k(y)}{\abs{x-y}^{2p+2\alpha_k +2z}}\, dy.$$
			Let us now fix a number $m \in \N \cup \{0\}$ and an index $j\in \{1,\ldots,N\}$ and subsequently consider the limit
			\begin{equation}\label{residue_limit_even} \lim\limits_{z\to -m-p-\alpha_j} (z+m+p+\alpha_j)^2\,F(z).\end{equation}
			On the one hand, the above limit is zero, thanks to Proposition~\ref{prop_F_zero}. On the other hand, we can evaluate the limit explicitly by residue calculus. Let us compute the limit above by breaking the sum into its $N$ components and see their contributions.

            It is straightforward to see that for each $k=1,\ldots,N$ with $k\neq j$, there holds:
			\begin{equation}\label{residue_limit_even_1} \lim\limits_{z\to -m-p-\alpha_j} (z+m+p+\alpha_j)^2\,\Gamma(z+1+\alpha_k)\Gamma(z+\alpha_k+p) =0,\end{equation}
			because the Gamma functions on the right-hand side do not have a pole in the limit, thanks to $\alpha_k \neq \alpha_j$ and $\alpha_k \in (0,1)$. Thus, the only term in the summand that could contribute to the limit would be the term $k=j$. For this term, both the Gamma functions have a simple singularity there and as such we have that
			\begin{equation}
            \begin{split}
                &\quad \, \lim\limits_{z\to -m-p-\alpha_j} (z+m+p+\alpha_j)^2\,\Gamma(z+1+\alpha_j)\Gamma(z+p+\alpha_j)\\
                &=\Big( \lim\limits_{z\to -m-p-\alpha_j} (z+m+p+\alpha_j)\Gamma(z+1+\alpha_j)\Big) \Big( \lim\limits_{z\to -m-p-\alpha_j} (z+m+p+\alpha_j)\Gamma(z+p+\alpha_j)\Big)\\
                &=\textrm{Res}(\Gamma;-m-p+1)\,\textrm{Res}(\Gamma;-m),
            \end{split}
                \end{equation}
			where we recall that for each  $\ell\in \N\cup \{0\}$, the notation $\textrm{Res}(\Gamma,-\ell)$, $\ell\in \N\cup \{0\}$ stands for the residue\footnote{Recall that if $c$ is a simple pole of the function $f$, the residue of $f$ is given by $\mathrm{Res}(f;c)=\displaystyle\lim_{z\to c}(z-c)f(z)$.} of Gamma function $\Gamma(z)$ at the point $z=-\ell$ that equals $\frac{(-1)^\ell}{\ell!}$. Therefore, 
			\begin{equation}\label{residue_limit_even_2} \lim\limits_{z\to -m-p-\alpha_j} (z+m+p+\alpha_j)^2\,\Gamma(z+1+\alpha_j)\Gamma(z+p+\alpha_j) =\frac{(-1)^{p-1}}{m!(m+p-1)!}.\end{equation}
			By combining \eqref{residue_limit_even_1}--\eqref{residue_limit_even_2} into the limit \eqref{residue_limit_even} and recalling that the answer must be zero, we obtain
			\begin{equation}
				\int_{\R^n\setminus \mathcal O} v_j(y)\, |x-y|^{2m}\,dy =0 \quad \text{for all $m\in \N\cup\{0\}$ and all $j=1,\ldots,N.$} 
			\end{equation}
			In particular, since $v_j$'s all vanish on the set $\mathcal O$, we obtain that
			\begin{equation}
				\label{residue_limit_even_final}
				\int_{\R^n} v_j(y)\, |x-y|^{2m}\,dy =0 \quad \text{for any $x\in \omega$, $m\in \N\cup\{0\}$ and $j=1,\ldots,N.$} 
			\end{equation}
			We will hold onto this and move on to the singularity analysis in the odd dimension.\\
			
			\noindent{\bf \em II.II. case of odd dimension.} We will assume now that $n=2p-1$ for some $p\in \N$ so that
			$$ F(z)= \sum_{k=1}^N \frac{4^{\alpha_k+z}}{\pi^{p-\frac{1}{2}}} \, \frac{\Gamma\LC z+1+\alpha_k\RC  \Gamma \LC z+\alpha_k +p-\frac{1}{2}\RC}{\Gamma(-\alpha_k)\,\Gamma(1+\alpha_k)} \int_{\R^n\setminus \mathcal{O}} \frac{v_k(y)}{\abs{x-y}^{2p-1+2\alpha_k +2z}}\, dy.$$
			Let us now fix a number $m \in \N \cup \{0\}$ and an index $j\in \{1,\ldots,N\}$ and subsequently consider the limit
			\begin{equation}\label{residue_limit_odd} \lim\limits_{z\to -m-1-\alpha_j} (z+m+1+\alpha_j)\,F(z).\end{equation}
			As in the previous case, on the one hand, the above limit is zero, thanks to Proposition~\ref{prop_F_zero}. On the other hand, we can evaluate the limit explicitly by residue calculus.

            It is straightforward to see that for each $k=1,\ldots,N$ with $k\neq j$, there holds:
			\begin{equation}\label{residue_limit_odd_1} \lim\limits_{z\to -m-1-\alpha_j} (z+m+1+\alpha_j)\,\Gamma \LC z+1+\alpha_k\RC \Gamma\LC z+\alpha_k+p-\frac{1}{2}\RC =0,\end{equation}
			because the Gamma functions on the right-hand side do not have a pole at the limit, thanks to the fact that $\alpha_k \neq \alpha_j$ as well as \eqref{exp_condition_alpha} and $\alpha_k \in (0,1)$. Thus, the only term in the summation that could contribute to the limit would be the term $k=j$. For this term, we note that
			\begin{equation}\label{residue_limit_odd_2}
				\begin{split}
					&\quad\, \lim\limits_{z\to -m-1-\alpha_j} (z+m+1+\alpha_j)\,\Gamma(z+1+\alpha_j)\Gamma\LC z+\alpha_j+p-\frac{1}{2}\RC \\
					&=\textrm{Res}(\Gamma;-m)\, \Gamma\LC -m+p-\frac{3}{2}\RC\\
					&=\frac{(-1)^m}{m!} \underbrace{\Gamma\LC -m+p-\frac{3}{2}\RC}_{\text{Nonzero}}
				\end{split}
			\end{equation}
			By combining \eqref{residue_limit_odd_1}--\eqref{residue_limit_odd_2} into the limit \eqref{residue_limit_odd} and recalling that the answer must be zero, we obtain
			\begin{equation}
				\int_{\R^n\setminus \mathcal O} v_k(y)\, |x-y|^{2m-2p+3}\,dy =0, \quad \text{for all $m\in \N\cup\{0\}$ and all $k=1,\ldots,N.$} 
			\end{equation}
			In particular, since $p\geq 1$ and since $v_k$'s vanish on the set $\mathcal O$, we obtain that
			\begin{equation}
				\label{residue_limit_odd_final}
				\int_{\R^n} v_k(y)\, |x-y|^{2m+1}\,dy =0, \quad \text{for any $x\in \omega$, $m\in \N\cup\{0\}$ and $k=1,\ldots,N.$} 
			\end{equation}
			We conclude this section by noting that our singularity analysis of $F(z)$ together with Proposition~\ref{prop_F_zero} has given us the two equations \eqref{residue_limit_even_final} and \eqref{residue_limit_odd_final} depending on the parity of the dimension $n$. The proof of Theorem~\ref{thm_ent_smooth} now follows immediately from the next proposition.

			\subsection{Support theorems for spherical mean transform}
			
			We aim to complete the proof of Theorem~\ref{thm_ent_smooth} via the following proposition.
			\begin{proposition}\label{prop_reduction from SM transform}
				Let $\omega \subset \R^n$ be a nonempty bounded open set and let $v\in C^{\infty}(\mathbb R^n)$ be identical to zero on $\omega$. Assume that there exist constants $C,\rho>0$ and $\gamma>1$ such that 
				\begin{equation}
					\label{v_super_decay}
					|v(x)| \leq C\,e^{-\rho |x|^\gamma} \quad \text{for any } x\in \R^n.
				\end{equation}
				Moreover, assume that given each $x\in \omega$ and each $m\in \N\cup \{0\}$, there holds
				\begin{equation}\label{moments} 
					\begin{cases}
						\int_{\R^n} v(y)\, |x-y|^{2m}\,dy =0,  &\text{if $n$ is even,}\\
						\int_{\R^n} v(y)\, |x-y|^{2m+1}\,dy =0,  &\text{if $n$ is odd}.
					\end{cases}
				\end{equation}
				Then, $v$ must vanish identically on $\R^n$.
			\end{proposition}

			The proof of the above proposition relies on the spherical mean transform, which can be found in  \cite[Theorem 3.3]{Quinto} (see also \cite[Theorem 1.1, Remark 1.3]{RAMM20021033}). Let us collect the result about the spherical mean transform in the next lemma for readers' convenience. 
			
			\begin{lemma}[Spherical mean transform]\label{Lem: SMT}
				Adopting all assumptions in Proposition \ref{prop_reduction from SM transform}, let us consider the spherical mean transform of $v\in C^\infty(\R^n)$ by
				\begin{equation}
					\mathrm{SM}v(x,t):=\int_{\mathbb{S}^{n-1}}v(x-t\theta)\, dV_{\mathbb{S}^{n-1}}(\theta), \quad t>0,
				\end{equation}
				where $\mathbb S^{n-1}$ is the unit sphere in $\R^n$ and $dV_{\mathbb S^{n-1}}$ stands for its surface measure. Suppose that $\mathrm{SM}v(x,t)=0$ for all $(x,t)\in \omega\times (0,\infty)$. Then $v\equiv 0$ on $\R^n$.
			\end{lemma}

			\begin{proof}[Proof of Proposition \ref{prop_reduction from SM transform}]
				We will only prove the proposition for the case that $n$ is even. The case of odd dimensions $n\geq 2$ can be obtained by using similar arguments. First, let us fix $x\in \omega$ and observe that by our hypothesis (and since $n$ is even) there holds
				\begin{equation}\label{iden_even}
					\int_{\R^n} v(x-y)\, |y|^{2m}\,dy =0 \quad \forall\, m\in \N \cup \{0\}. 
				\end{equation}
				Fixing $x\in \omega$, and recalling that $v$ vanishes on $\omega$. Let us define the function $h\in C^{\infty}([0,\infty))$ via spherical mean transforms 
				\begin{equation}\label{sphere_mean} 
					h(t)= t^{n-1}\int_{\mathbb S^{n-1}} v(x-t\theta)\,dV_{\mathbb S^{n-1}}(\theta) \qquad t\geq 0.
				\end{equation}
			 Let us also define $f\in L^{\infty}(\R)$ via 
				$$ f(t) = \begin{cases}
					h(t) \quad &\text{if $t \geq0$},\\
					0   \quad &\text{if $t<0$}.
				\end{cases}
				$$
				As the function $v$ satisfies the super-exponential decay \eqref{v_super_decay}, it is straightforward to see that there is a $C'>0$ and $\rho' \in (0,\rho)$ such that 
				\begin{equation}
					\label{f_decay}
					|f(t)| \leq C' e^{-\rho' |t|^{\gamma}}.
				\end{equation}
				In terms of the function $f$, identity \eqref{iden_even} now reduces to 
				\begin{equation}
					\label{f_iden}
					\int_\R f(t) \,t^{2m}\,dt =0, \quad \text{for all } m\in \N\cup \{0\}.
				\end{equation}
				Let us next define $\mathcal F_f \in C^{\infty}(\C)$ to be the Fourier--Laplace transform of $f$ defined by 
				\begin{equation}
					\label{Fourier_Laplace}
					\mathcal F_f(z) := \int_{\R} f(t) e^{-\textrm{i}tz}\,dt, \quad z\in \C.
				\end{equation}
				Observe that the super-exponential decay \eqref{f_decay} of the function $f$ makes the above definition well-defined and makes $\mathcal F_f$ to be an entire function. The condition \eqref{f_iden} and analytic continuation imply that all the even order derivatives of the function $\mathcal F_f(z)$ must vanish at the origin, that is to say
				$$
				\LC \frac{\partial^{2m}}{\partial z^{2m}} \mathcal F_f
				\RC (0)=0, \quad \text{for all } m\in \N \cup \{0\}.
				$$
				As $\mathcal F_f$ is entire, its power series expansion at zero is absolutely convergent everywhere and therefore the latter identity implies that $\mathcal F_f$ is an odd function. Using the inverse Fourier transform 
				$$ f(t) = \frac{1}{2\pi}\,\int_\R \mathcal F_f(\xi)\, e^{\textrm{i}\xi t}\,d\xi,$$
				we deduce that $f$ must also be an odd function. As it vanishes for all $t<0$, it must vanish everywhere. Recalling the definition of $f$ we obtain next that the spherical means of the function $v$ must vanish on any sphere whose center lies on the set $\omega$. By Lemma \ref{Lem: SMT}, we can conclude that the function $v$ must vanish everywhere in $\R^n$.
			\end{proof}

			\section{Proofs for inverse problems \ref{Q:IP1}--\ref{Q:IP2}}\label{sec: IP}
			
			In this section, we prove both Theorems \ref{Thm: global uniqueness 1} and \ref{Thm: global uniqueness 2}. The proof of both Theorems crucially relies on special cases of our entanglement principle. 
        		
			\begin{proof}[Proof of Theorem \ref{Thm: global uniqueness 1}]
			Recalling Proposition~\ref{prop_solvability}, we choose any nonzero $F \in C^{\infty}_0(\Omega_e)$ satisfying the (exterior) orthogonality condition \eqref{phi_ortho} so that equation \eqref{equ: main_solv} (with $A$ replaced by $A_1$) admits some solution $u^{(1)}\in W^{2,2}_{\delta_0}(\R^n)$. Observe that there are infinitely many such $F$'s as $\mathcal K_{A_1}$ is finite-dimensional (see also Lemma~\ref{lem_ortho}). Note also that by Proposition~\ref{prop_solvability}, we have $u^{(1)}\in W^{k,2}_{\delta_0}(\R^n)\cap W^{k,2}_{-\delta_0-2}(\R^n)$ for any $k\in \{0\}\cup \N$.  Let us now consider
				\begin{equation}\label{def_fg}  
					(f,g) := (u^{(1)}|_{\Omega_e}, F) =(u^{(1)}|_{\Omega_e}, (L_{A_1} u^{(1)})|_{\Omega_e})\in \mathcal C_{A_1}.\end{equation}
				As $\mathcal C_{A_1}=\mathcal C_{A_2}$, it follows from the definition of $\mathcal C_{A_2}$ that there exists $u^{(2)}\in W^{2,2}_{\delta_0}(\R^n)$ with 
				$$ L_{A_2} u^{(2)} =0 \quad \text{in $\Omega$},$$
				such that there holds
				$$ \left. u^{(2)}\right|_{\Omega_e}=\left. u^{(1)}\right|_{\Omega_e}=f \quad \text{and}\quad (L_{A_2}u^{(2)})|_{\Omega_e}= g=F.$$
				As $L_{A_2}u^{(2)}=0$ in $\Omega$ and as it equals $F$ in $\Omega_e$ and finally as $F$ is compactly supported in $\overline{\Omega_e}$, we conclude that
				$$ L_{A_2} u^{(2)} = F \quad \text{in $\R^n$},$$
                holds in the $L^2_{\delta_0+2}$-sense. We note here that since $F\in C^{\infty}_0(\Omega_e)$ and since $u^{(2)}\in W^{2,2}_{\delta_0}(\R^n)$ we have that $u^{(2)}\in W^{k,2}_{\delta_0}(\R^n)\cap W^{k,2}_{-\delta_0-2}(\R^n)$ for all $k\in \N$, thanks to Proposition~\ref{prop_solvability}. This also implies that the equation is satisfied point-wise.
                Next, introducing
                $$u:=u^{(2)}-u^{(1)} \in W^{k,2}_{\delta_0}(\R^n)\cap W^{k,2}_{-\delta_0-2}(\R^n), \quad k=0,1,\ldots,$$ 
                we note that
			\begin{equation}\label{u_2_diff_eq}
				u|_{\Omega_e}=0 \quad \text{and}\quad  L_{A_2}u= \nabla\cdot(\underbrace{(A_{2}-A_{1}}_{\textrm{Henceforth $\widetilde{A}$}})\nabla u^{(1)}) \quad \text{in $\R^n$}.
			\end{equation}
                Recalling that $A_{1}|_{\Omega_e}=A_{2}|_{\Omega_e}=\mathds{1}_{n\times n}$ and that $P_0$ is given by \eqref{P_0_def}, we obtain 
				\begin{equation}
					\label{u_12_eq}
					u|_{\Omega_e}=0 \quad \text{and} \quad P_{0} u= 0 \quad \text{in $\Omega_e$}. 
				\end{equation}
				Noting that the functions $\{p_k\}_{k=1}^N\subset C^{\infty}_0(\R^n)$ satisfy \eqref{p_k_def} and thus are each a nonzero constant in an open neighbourhood of $\overline{\Omega}$, we can now apply our entanglement principle, namely Theorem~\ref{thm: ent} with the choice $u_k:=p_k\,u$ for $k=1,\ldots,N$, in its statement and with $\mathcal O=\mathcal U\setminus \overline{\Omega}$ where $\mathcal U$ is as in \eqref{p_k_def}.  The entanglement principle now implies that the function $u$ must vanish everywhere on $\R^n$. Consequently, $P_{0}u$ must also vanish globally and thus in particular,
				\begin{equation}\label{A_eq_0}
				    \nabla\cdot\left(\widetilde{A}(x)\nabla u^{(1)} \right) =0 \quad \text{in $\R^n$},
				\end{equation}
                where we recall that $\widetilde{A}(x)=A_{2}(x)-A_{1}(x)$. 
                
                To summarize, we have proved that given any $F\in C^{\infty}_0(\Omega_e)$ that satisfies \eqref{phi_ortho} and any solution $u^{(1)}\in W^{2,2}_{\delta_0}(\R^n)$ that solves \eqref{equ: main_solv} (with $A=A_{1}$), equation \eqref{A_eq_0} must be satisfied globally. As $u^{(1)}+\zeta$ is also a solution to equation \eqref{equ: main_solv} for any $\zeta\in \mathcal K_{A_1}$ (recalling the definition \eqref{K_A_def}), and as $\widetilde{A}$ is real-valued, we  deduce in particular that
                \begin{equation}
                	\label{A_eq_0_1}
                	   \nabla\cdot\left(\widetilde{A}(x)\nabla \overline{\zeta} \right) =0 \quad \text{in $\R^n$}, \quad \forall\, \zeta \in \mathcal K_{A_1}.
                \end{equation}
                Next, let $h \in C^{\infty}_0(\Omega)$ be an arbitrary function. By \eqref{A_eq_0_1} we have that 
           		   \begin{equation}
           		 	\label{A_eq_0_2}
           		 	\left( \nabla\cdot\left(\widetilde{A}(x)\nabla h \right), \overline{\zeta}\right)_{L^2(\R^n)} =0 \qquad \forall\, \zeta \in \mathcal K_{A_1}.
           		 \end{equation}
           		 Using this equality together with Proposition~\ref{prop_solvability}, we deduce that there exists a solution $w\in W^{2,2}_{\delta_0}(\R^n)\cap W^{2,2}_{-2-\delta_0}(\R^n)$ to the equation 
           		 $$L_{A_1} w =  \nabla\cdot\left(\widetilde{A}(x)\nabla h \right)  \quad \text{in $\R^n$}.$$              
         Recalling equation \eqref{A_eq_0} and the fact that $h$ is compactly supported in $\Omega$, we deduce that
              \begin{equation}
             	\begin{split}
             		0&=\LC \nabla\cdot\left(\widetilde{A}(x) \nabla u^{(1)} \right),\overline{h}\RC_{L^2(\Omega)}=\LC u^{(1)}, \nabla\cdot\left(\widetilde{A}(x) \nabla \overline{h}\right)\RC_{L^2(\Omega)}\\
             		&=\LC u^{(1)}, \nabla\cdot\left(\widetilde{A}(x) \nabla \overline{h}\right)\RC_{L^2(\R^n)}\\
             		&=\LC u^{(1)},\overline{L_{A_1}w}\RC_{L^2(\R^n)}=\LC L_{A_1}u^{(1)},\overline{w}\RC_{L^2(\R^n)}\\
             		&=\LC F,\overline{w}\RC_{L^2(\Omega_e)}.
             	\end{split}
             \end{equation} 
				As $F\in C^{\infty}_0(\Omega_e)$ is any function that satisfies \eqref{phi_ortho}, we deduce via Lemma~\ref{lem_ortho_cond} that  
				\begin{equation}\label{w_eq_1}
					w|_{\Omega_e} = \zeta|_{\Omega_e}, \quad \text{for some $\zeta \in \mathcal K_{A_1}$}.
				\end{equation}
                Defining the function $\widetilde w \in W^{2,2}_{\delta_0}(\R^n)\cap C^{\infty}(\R^n)$ via $\widetilde w = w-\zeta$, we obtain 
                that
				\begin{equation}\label{tilde_w_eq} 
					L_{A_1}\widetilde w = \nabla\cdot\left(\widetilde{A}(x)\nabla h \right) \quad \text{in $\R^n$}\end{equation} 
				and that $\widetilde{w}=0$ in $\Omega_e$, and therefore, recalling that $h\in C^{\infty}_0(\Omega)$, there holds
				$$ 
				\left. (P_{0}\widetilde w) \right|_{\Omega_e}= \left. \widetilde{w}\right|_{\Omega_e}=0.
				$$
				We can now apply our entanglement principle, namely Theorem~\ref{thm: ent} with the choice $u_k=p_k\widetilde w$ for $k=1,\ldots,N$, in its statement and with $\mathcal O=\mathcal U$ where $\mathcal U$ is as in \eqref{p_k_def}). Noting that $p_k$'s are nonzero on $\Omega$, the theorem yields that $\wt w =0 $ in $\Omega$ so that $\widetilde w$ must vanish on $\R^n$, where we used \eqref{w_eq_1}.
	Hence, we conclude via \eqref{tilde_w_eq} that 
      \begin{equation}\label{h_final}\nabla\cdot\left(\widetilde{A}(x) \nabla h\right) =0 \quad \text{in $\R^n$}.\end{equation}
     
     To summarize, we have shown that given any $h\in C^{\infty}_0(\Omega)$, the above equation must be satisfied globally. Next, let us choose an arbitrary point $x_0 \in \Omega$ and let $B_\delta(x_0)$ be the closed  ball of radius $\delta>0$ centered at $x_0$ where $\delta\in (0,1)$ is sufficiently small, so that $B_\delta(x_0)\subset \Omega$. Let $\chi_0:\R\to \R$ be a nonnegative smooth cutoff function such that $\chi_0(t):=\begin{cases}
         1 & \text{ for } |t|\leq \frac{1}{4} \\
         0 & \text{ for } |t|\geq \frac{1}{2}
     \end{cases}$. Subsequently, let us consider for each $\lambda\in \R$ and each vector $\xi=(\xi_1,\ldots,\xi_n) \in \R^n$, the function
     $$ h_{\lambda,\xi}(x)= e^{\textrm{i} \lambda \xi\cdot (x-x_0)} \chi_0(\delta^{-1}|x-x_0|).$$
     Substituting $h=h_{\lambda,\xi}$ into equation \eqref{h_final}, evaluating it at $x_0$, dividing by $\lambda^2$ and finally taking the limit as $\lambda \to \infty$, we conclude that there holds
     $$ \sum_{j,k=1}^n \widetilde{A}^{jk}(x_0) \,\xi_j \xi_k=0 , \quad \forall\, \xi \in \R^n,$$
     and finally (due to symmetry) that $\widetilde{A}(x_0)=0$. The theorem is now proved since $x_0\in \Omega$ is arbitrary. 
			\end{proof}

			Now, we turn to prove Theorem \ref{Thm: global uniqueness 2}. We present a different perspective based on the Runge approximation property for the sake of interest. We first show the Runge approximation property.

			\begin{proof}[Proof of Theorem \ref{Thm: Runge}]
				Consider the set
				\begin{equation}
					\mathcal{R}:=\left\{ u_f -f : \, f\in C^\infty_0(W)\right\},
				\end{equation}
				where the notation $u_f \in H^{s_N}(\R^n)$ stands for the unique solution to \eqref{equ: main 2}. It suffices to show that $\mathcal{R}$ is dense in $L^2(\Omega)$.
				The proof is standard by using the duality argument with the Hahn-Banach theorem. It suffices to show that if $v\in L^2(\Omega)$ satisfies $\LC  v, w \RC_{L^2(\Omega)}=0$ for any $w\in \mathcal{R}$, then $v\equiv0$. To show this, let $v$ be such a function, that is to say,
				\begin{equation}\label{eq:11111111}
					\LC v,u_{f}-f\RC_{L^2(\Omega)} =0,\quad \text{for any }f\in C_{0}^{\infty}(W).
				\end{equation}
				Let $\phi\in\widetilde{H}^{s_N}(\R^n)$ be the unique solution of $P_q \phi=v$ in $\Omega$ subject to vanishing exterior Dirichlet data, namely $\phi|_{\Omega_e}=0$. Next, observe that  for any $f\in C_{0}^{\infty}(W)$, there holds

				\begin{equation}\label{eq:formal_adjoint_relation}
					B_{q}(\phi,f)=B_{q}(\phi,f-u_{f})=\LC v,f-u_{f}\RC _{L^2(\Omega)}
				\end{equation}
				where $B_q(\cdot,\cdot)$ is the sesquilinear form given by \eqref{B_q bilinear} and we have used the facts that $u_{f}$ is the solution of \eqref{equ: main 2}, and $\phi\in\widetilde{H}^{s}(\Omega)$ (recall that $\phi|_{\Omega_e}=0$).
				Applying \eqref{eq:11111111} and \eqref{eq:formal_adjoint_relation} together with the fact that $q$ is zero outside $\Omega$, we deduce that 
				\[
				\big( \sum_{k=1}^N b_k (-\Delta)^{s_k}\phi,f\big)_{L^2(\mathbb{R}^{n})}=0\mbox{ for any }f\in C_{0}^{\infty}(W).
				\]
				Thus, the function $\phi\in \wt  H^{s_N}(\Omega)$  satisfies
				\[
				\phi=\sum_{k=1}^N b_k (-\Delta)^{s_k}\phi =0 \text{ in }W.
				\]
				Thanks to the entanglement principle of Theorem \ref{thm: ent} again, we obtain $\phi\equiv0$ and then $v\equiv0$. Note that the application of the entanglement principle is justified here as $\phi=0$ in $\Omega_e$, and thus in particular satisfies the super-exponential decay condition automatically.
			\end{proof}

			\begin{proof}[Proof of Theorem \ref{Thm: global uniqueness 2}]
				We follow the same argument as the proof of \cite[Theorem 1.1]{GSU20}. 	
				If $\left. \Lambda_{q_{1}}f\right|_{W_{2}}=\left.\Lambda_{q_{2}}f\right|_{W_{2}}$ for any $f\in C_{0}^{\infty}(W_{1})$, where $W_{1}$ and $W_{2}$ are open subsets of $\Omega_{e}$, by the integral identity \eqref{eq:integral identity}, we have 
				\[
				\int_{\Omega}(q_{1}-q_{2})u_{1}\,\overline{u_{2}}\, dx=0,
				\]
				where $u_{j}\in H^{s_N}(\R^{n})$ solves $P_{q_j}u_{j}=0$ in $\Omega$ with $u_{j}$ having exterior values $f_{j}\in C_{0}^{\infty}(W_{j})$, for $j=1,2$.
				
				Given an arbitrary $\phi\in L^{2}(\Omega)$ and by using the Runge approximation of Theorem \ref{Thm: Runge}, there exists two sequences of functions $\big( u_{\ell}^{1}\big)_{\ell\in \N}$, $\big(u_{\ell}^{2}\big)_{\ell\in \N}\subset  H^{s_N}(\mathbb{R}^{n})$ that fulfill
				\begin{align*}
					& P_{q_1} u_{\ell}^{1}=P_{q_2} u_{\ell}^{2}=0\text{ in \ensuremath{\Omega}},\\
					& \supp\big(u_{k}^{1}\big)\subseteq\overline{\Omega_{1}}, \quad \supp\big(u_{\ell}^{2}\big)\subseteq\overline{\Omega_{2}},\\
					& \left. u_{\ell}^{1}\right|_{\Omega}=\phi+r_{\ell}^{1},\quad \left. u_{\ell}^{2}\right|_{\Omega}=1+r_{\ell}^{2},
				\end{align*}
				where $\Omega_{1}$, $\Omega_{2}\subset \R^n$ are two open sets containing $\Omega$, and $r_{\ell}^{1},r_{\ell}^{2}\to0$ in $L^{2}(\Omega)$ as $\ell\to\infty$. Plug these solutions into the integral identity and pass the limit as $\ell\to\infty$, then we infer that 
				\[	
				\int_{\Omega}\LC q_{1}-q_{2}\RC \phi \, dx=0.
				\]
				As $\phi\in L^{2}(\Omega)$ is arbitrary, we can conclude that $q_{1}=q_{2}$ in $\Omega$. This concludes the proof.
			\end{proof}

			\medskip 
			
			\noindent\textbf{Acknowledgment.} 
			Y.-H. Lin is partially supported by the Ministry of Science and Technology Taiwan, under projects 113-2628-M-A49-003 and 113-2115-M-A49-017-MY3.

			\bibliography{refs} 
			
			\bibliographystyle{alpha}
			
		\end{document}